\documentclass[10pt]{article}

\usepackage{caption}

\usepackage{amsmath,amsfonts,amssymb,amsthm,mathrsfs}
 \usepackage{bm}
\usepackage{graphicx}
\usepackage{stmaryrd}
\usepackage{color}

\usepackage{hyperref}

 \newtheorem{definition}{Definition}[section]
\newtheorem{remark}{Remark}[section]
\newtheorem{example}{Example}[section]

\newtheorem{theorem}{Theorem}[section]
\newtheorem{lemma}{Lemma}[section]
\newtheorem{proposition}{Proposition}[section]

\newtheorem{corollary}{Corollary}[section]

\renewcommand{\d}{\text{\rm{d}}}
\newcommand{\equlaw}{\stackrel{(d)}{=}}
\newcommand{\Var}{\text{\rm{\bf Var}}}
\newcommand{\bo}{\bm{\omega }}
\newcommand{\bmu}{\bm{\mu }}
\newcommand{\cG}{c_{\text{\rm{inf}}}}
\newcommand{\cB}{c_{\text{\rm{Bin}}}}
\newcommand{\x}{{\bf x}}
\newcommand{\argmin}{\text{\rm{argmin}}}
\newcommand{\U}{{\bf U}}

\newcommand{\Leb}[1][d]{\mathscr  L^{{#1}}}
\newcommand{\be}{{\bf e}}
\newcommand{\barU}{\overline{\U}}
\newcommand{\bS}{{\mathbb  S}}
\newcommand{\K}{ {\bf K}}

 \newcommand{\odd}{\text{\rm{ odd}}}
\renewcommand{\o}{{\bf 0}}
\newcommand{\q}{{\bf q}}

\renewcommand{\t}{{\bf t}}
\newcommand{\y}{{\bf y}}

\newcommand{\I}{{\bf I}}
\newcommand{\J}{{\bf J}}
\newcommand{\leqio}{\stackrel{\text{\rm{\bf i.o.}}}{\leqslant }}

\newcommand{\s}{{\bf s}}
\newcommand{\z}{{\bf z}}

\numberwithin{equation}{section}

\title{ Diophantine   Gaussian excursions and random walks }
\begin{document}

%
%
\maketitle

{\bf Abstract}
We establish general asymptotic upper and lower bounds for the volume variance of Euclidean Gaussian nodal excursions in terms of the random walk associated to the spectral measure. These bound are sharp in several situations, and under mild assumptions,  the variance is at least linear.

To obtain sublinear variances, we focus on the case
where the spectral measure is purely atomic, and show that  the associated irrational random walk on the multi-dimensional torus comes back more often close to $0$ when the atoms are well approximable by rational tuples. Hence the excursion behaviour strongly depends on the {\it diophantine properties} of the atoms, i.e. on the quality of approximation of the atom locations by   rationals. The volume variance has   fluctuations which power can be arbitrarily close from the maximum $2d$ (quadratic fluctuations), whereas  if the atoms are badly approximable the excursion is strongly hyperuniform, meaning the variance asymptotic power is  minimal, $(d-1)$, corresponding to the window boundary measure. 
Also, given any reasonable variance asymptotic behaviour, there are uncountably many sets of spectral atoms that realise it.

The versatility of the variance formula is  illustrated by other examples where the spectral measure support can have higher dimension, in particular it is able to capture the variance cancellation phenomenon of Gaussian random waves, and it also yields that there are no hyperuniform isotropic Gaussian excursions.\\


{\bf AMS 2010} {60G15, 60G50, 11J13, 34L20  }

 \newpage

 \tableofcontents

\section{Introduction}

 The primary motivation of this article is to study the variance of the excursion volume  for Euclidean stationary Gaussian fields, and exhibit a class of models   that realise any prescribed asymptotic variance behavior. It turns out that this can only be achieved by spectral measures with a low dimensional support, hence we consider  measures with a finite support. This investigation  requires to study a random walk which  behaviour depends on the diophantine properties of the spectral atoms, i.e. on the quality of approximation of the atom locations by rationals.
 To conduct this program, we establish two unrelated results which are  of independent interest, corresponding to Sections \ref{sec:exc} and \ref{sec:RW}, which are the main results of this paper.  They are then combined in Section \ref{sec:intro-Gauss-dioph}.
 
 The first result, Theorem \ref{thm:general-var}, deals with the volume of general Euclidean Gaussian fields excursions. The main finding is that the variance magnitude is strongly related to the probability of the associated random walk to return around $0.$ The second result, Theorem \ref{thm:dioph-rw}, contains bounds for random walks  with irrational increments.  The combination of those two results yields variance asymptotics for diophantine Gaussian excursions, as detailed in Section \ref{sec:intro-Gauss-dioph}, culminating with Theorem \ref{thm:main-intro}.
 
 The results about diophantine random walks are of independent interest and can be projected onto the torus, adding some uniform estimates to the existing literature (see Section \ref{sec:intro-walk}).  
 
 The fact that  Theorem \ref{thm:general-var} has a more general scope and can be applied in various situations is illustrated by the variance cancellation phenomenon for a fundamentally different model, the Gaussian random wave, see Section \ref{sec:intro-waves}.\\

  \subsection{Gaussian excursions volume variance}
  \label{sec:intro-Gauss}
  
  The main actors of this article are  centred stationary real Gaussian random fields $\{X(\t);\t\in \mathbb{R}^{d}\}$, which law is invariant under translations of $\mathbb{R}^{d}$. See the monograph \cite{AT07} for a comprehensive exposition of main properties and  fundamental results about Gaussian fields and their geometry.  It is known that they are completely characterised by their reduced covariance function
\begin{align*}
C(\t)=\mathbb{E}(X(0)X(\t)),\t\in \mathbb{R}^{d},
\end{align*} 
or by their {\it spectral measure}, i.e. the unique finite symmetric measure $\bmu $ on $\mathbb{R}^{d}$ such that $C$ admits the representation
\begin{align}
\label{eq:def-spectral}
C(\t)=\int_{\mathbb{R}^{d}}e^{-i  \t \cdot \x   }\bmu (d\x),\t\in \mathbb{R}^{d},
\end{align}
 where $ \cdot    $ denotes the standard scalar product. 
 In all the paper, quantities related to vectors in $\mathbb{R}^{d}$ are denoted in bold ($\t,\bmu,\x,\bo,$ etc...).

Excursions of Gaussian processes on the real line have  often been studied through their number of crossings with the axis, see \cite{Kac43,CL67,Cuz,Slud91,KL97} or the survey \cite{Kratz-survey}. Elementary considerations yield that the average number of crossings on an interval is proportional to the length of the interval. Furthermore, if $\bmu $ contains more than one (symmetrised) atom, the variance of the number of crossings is quadratic \cite{Lac20,ABF}. We  focus here on the Lebesgue measure of the nodal excursions 
\begin{align*}
\{X>0\}=\{\t\in \mathbb{R}^{d}:X(\t)>0\}.
\end{align*}
Here again, the field centering and an application of Fubini's theorem yields that the expectation is proportional to the volume:
\begin{align*}
\mathbb{E}(\Leb[d](A\cap \{X>0\}))=\frac{\Leb[d](A)}{2},A\subset \mathbb{R}^{d}
\end{align*}where $ \Leb$ is the $ d-$dimensional Lebesgue measure. We give in Section \ref{sec:exc} general upper and lower bounds for the variance of the excursion volume
\begin{align*}
V_{\bmu}(T)=\Var(\Leb(\{X>0\}\cap B_{d}(0,T)))
\end{align*} 
where $B_{d}(0,T)$ is the centered ball with radius $T.$ These bounds imply in particular that if $X$ is isotropic, or more generally if $\bmu$'s support has dimension $\geqslant 1$ and spans the whole space, the volume has at least linear variance, i.e. larger than $T^{d}$ (see Section \ref{sec:structure}).

If on the other hand $\bmu$'s support is finite, a wide class of asymptotic behaviours are reachable.

\begin{corollary}Let  $\psi(q) $ a function from $\mathbb{N}^{*}$ to $ [0,1]$  decaying regularly  faster than $q^{-\frac{ 1+2d}{1+d}}$ (Definition \ref{def:regular}). 
There are uncountably many finite sets $\Sigma\subset \mathbb{R}^{d} $  such that if $\bmu$ is symmetric with support $\Sigma ,$
there are $0<c_{-}\leqslant c_{+}<\infty $ such that for $T>0$ sufficiently large 
\begin{align*}
c_{-} T^{2d}{\psi ^{-1}(T)^{-(1+2d)}}\leqio V_{\bmu}(T)\leqslant 
c_{+}T^{2d}{\psi ^{-1}(T)^{-(1+2d)}}\end{align*} 
where $\psi ^{-1}$ is the pseudo-inverse of $\psi $ (defined at \eqref{eq:pseudo-inv}) and
 $\leqio$ means that the inequality is true for a sequence $T_{k}\to \infty .$
\end{corollary}

By carefully choosing the function $\psi $, one can hence have any prescribed variance magnitude which has sufficiently regular variation.
We have a parametric model which achieves any reasonable asymptotic variance between the minimal {\it surface-scaling order}, in $T^{d-1}$,  and the maximal {\it quadratic order}, in $T^{2d}$(see Proposition \ref{cor:intro}
 for explicit choices of $\Sigma$ which yield any power law behaviour for the variance). 
The need for models that yield any prescribed variance asymptotics is explained in \cite{ChenTor18}, along with another such procedure based on  Fourier transforms.
When the  variance is sublinear (below $T^{d}$), the excursion is said to be  {\it hyperuniform}, contributing to the already large research body on the subject (see Section \ref{sec:structure} for more insights and motivation).

The result above is obtained with a Gaussian field which spectral measure $\bmu$ is of the form 
$$\bmu=\sum_{k=1}^{d} \sum_{i=0}^{m}\bar \delta _{\omega _{i}\be_{k}},\text{\rm{  where  }}\bar \delta _a=\frac{ 1}{2}(\delta _{a}+\delta _{-a}),a\in \mathbb{R}^{d},$$ and the  atoms $\omega _{i},i=1,\dots ,m$  are   $\psi $-approximable, i.e. roughly speaking such that for infinitely many $q=(q_{i})\in \mathbb{Z} ^{m}$, $\sum_{i}\omega _{i}q_{i} $ is $\psi (q)$-close to an integer (and this is not true for $\varphi <<\psi $), and the $\be_{k},1\leqslant k\leqslant d$ form a basis of $\mathbb{R}^{d}$. Let us give formal definitions, as they will be useful throughout the introduction and the paper: say that $\omega \in \mathbb{R}^{m}$ is 
\begin{align}
\label{eq:def-BA-intro}
\text{\rm{$\psi $-BA (Badly Approximable), if for some $r>0$}}&\\\notag
\vspace{-2cm}|p-\omega \cdot q | \geqslant 2\psi (q) \text{\rm{   for all  }}p\in \mathbb{Z} ,&\,q\in \mathbb{Z} ^{m}\setminus B_{m}(0,r)\\
\label{eq:def-WA-intro}
\text{\rm{  and  $\psi $-WA (Well approximable), if for some $c>0$}}&\\\notag
 | p-\omega \cdot q | <c\psi (q) \text{\rm{  for infinitely many  }}&p\in \mathbb{Z} , q\in \mathbb{Z} ^{m},q\equiv 1,
\end{align} 
where $q\equiv 1$ means that $\sum_{i=1}^{m}q_{i}$ is an odd number.
The proof  consists in (i) expressing the variance in terms of the behaviour around $0$ of the diophantine random walk which increment measure is $\bmu$ (see Theorem \ref{thm:general-var}) and (ii) studying this random walk with the help of results from diophantine approximation theory, see Theorem \ref{thm:dioph-rw}. Independently, we also consider Gaussian random waves (Theorem \ref{thm:var-cancel}) and short range fields (Proposition \ref{prop:integrable}) to illustrate the wide scope of this method.

More refined results from  diophantine approximation theory actually yield the quantity of tuples $(\omega_{i}) $ yielding a given asymptotics variance, and we build at Section \ref{sec:randomised} mixtures of such Gaussian models with a random support $\Sigma $ giving a prescribed asymptotic variance.\\

 The rest of the introduction is mostly illustrative, it presents some aspects and corollaries of the important theoretical results of this paper (Theorems \ref{thm:general-var} and \ref{thm:dioph-rw}). Their proofs necessarily call to subsequent sections. None of the rest of the current section is necessary to read or understand the rest of the paper.

  \subsection{Background and motivation}
  
  Properties of excursions and level sets of continuous random Gaussian functions have been studied  under many different instances. The zero set of a one-dimensional Gaussian stationary process is the subject of an almost century long line of research, starting with the seminal works of Kac \& Rice \cite{Kac43}, or Cram\'er \& Leadbetter \cite{CL67}, and followed by many other authors mainly interested by second order behaviour, see the major contributions by Cuzick, Slud, Kratz \& L\'eon, \cite{Cuz,Slud91,KL97}. In higher dimensions,  zeros of Gaussian entire functions  \cite{BKPV,NS11} and nodal sets of high energy Gaussian harmonics on a compact manifold  \cite{KKW13,Wigman10,MPRW} (and their Euclidean counterpart the Random Wave Model \cite{MRV}) have attracted a lot of attention from both physicists and mathematicians. Random trigonometric Gaussian polynomials, i.e. independent Gaussian coefficients multiplied by trigonometric monomials based on a fundamental frequency,  have also been studied in the asymptotics of the large degree, see for instance \cite{WigGra11} and references therein.  We propose here a  crucial modification of the spectral measure support: instead of taking   frequencies in a proportional relation, we choose finitely many frequencies which are incommensurable; this specificity allows for instance to reach all possible behaviours for the variance asymptotics (see Theorem \ref{thm:main-intro}).\\

A different approach to our results is through the lens of  hyperuniform   models, defined at section \ref{sec:structure}. In the last decades, physicists have put in evidence states of matter  intermediate between crystals and liquids, where the medium exhibits apparent disorder at the local scale, but  fluctuations are suppressed at large scales. 
This denotes in some sense a long-range compensation of the medium behaviour,  and is considered by physicists a {\it new state of matter}, see the works by S. Torquato and his co-authors   \cite{Tor18,TS03} that introduce the topic and expose the main tools and discoveries. Even though the focus was primarily on atomic measures,   this concept has been then generalised to other random measures, in particular  bi-phased random media \cite{Tor16,Tor16b}.  Such heterogeneous materials abound in nature and synthetic
situations. Examples include composite and porous media, metamaterials, biological media (e.g., plant and animal
tissue), foams, polymer blends, suspensions, granular media, cellular solids, colloids.

The Gaussian realm provides models for many types of phenomena, and the present work yields Gaussian hyperuniform random sets, i.e. which variance on a large window is asymptotically negligible with respect to the window volume (see Section \ref{sec:structure}). The model we present here shares some similarities with perturbed lattices, in the sense that the long range correlations are very strong, but its disorder state  is also one step above as one cannot write it as the (perturbed) repetition of a given pattern. It shares with quasi-crystals the property of {\it almost periodicity}, defined below, and exhibits a spectrum reminiscent of quasi-crystals, see Fig \ref{fig:1}. Any asymptotic variance can be achieved, yielding in particular hyperuniform models. According to the typology established in  \cite[6.1.2]{Tor18}, the model is  type-I hyperuniform for almost all choice of parameters; but uncountably many choices of the parameters will yield type-II or actually any type of hyperuniformity. We also give randomised versions of the model not involving diophantine parameters which exhibit different types of hyperuniformity.

As it turns out,  a  non-isotropic model is necessary to obtain a hyperuniform behaviour (Proposition \ref{prop:isotropic}).
We use the general variance formula of Theorem \ref{thm:general-var} to study Gaussian random waves in any dimension and prove a variance cancellation phenomenon.\\

 \subsection{Diophantine random walk on the torus}
\label{sec:intro-walk}

We derive in Section \ref{sec:RW} results about diophantine random walks on $\mathbb{R}^{d}$, which ultimately lead to variance estimates for Gaussian diophantine excursions. The current section discusses the connections with the existing literature for the diophantine  random walks on the torus, and is completely disconnected from the results about Gaussian fields discussed above.

Let $(\be_{1},\dots ,\be_{d})$ be a basis of $\mathbb{R}^{d}$, $m\geqslant 1,\bmu$ be a symmetric measure on $\mathbb{R}^{d}$ parametrised by its support $\bo=(\omega _{[k]})_{1\leqslant k\leqslant d}\in (\mathbb{R}^{m})^{d}$ via
\begin{align}
\label{mu:general}
\bmu=\frac{1}{d(m+1)}\sum_{k=1}^{d} \sum_{i=0}^{m}\bar\delta _{\omega _{[k],i}\be_{k}},
\end{align}
with $\omega _{[k],0}=1$ by convention, 
   and let $\overline{\U}_{n}$ be the corresponding random walk on the torus
\begin{align*}
\overline{\U}_{n}= \{\sum_{i=1}^{n}X_{i}\}
\end{align*}
where the $X_{i}$ are independent and identically distributed with law $\bmu$ and $\{\x \}=(\{x_{[k]}\})\in [0,1[^{d}$ is the fractional part in $\mathbb{R}^{d}$.   The index $1\leqslant k\leqslant d$ relating to different dimensions in $\mathbb{R}^{d}$ is written between brackets $[k]$, to avoid confusion with the subscript $i$, usually running over different frequencies on each dimension. 

 It is clear that if  $\bo$'s components are well approximable by rationals,  the same goes for the increments of the random walk, hence it is likely to come back closer to $0$ faster. The study of random walks on a group started on finite arithmetic groups with the works of Diaconis, Saloff-Coste, Rosenthal, Porod, (see references in \cite{Su98}) and results for such irrational random walks in the continuous settings were then achieved by Diaconis \cite{Diaconis88}, and finally Su \cite{Su98}, who gave the optimal speed of convergence   of the law of $\barU_{n}$ in an appropriate distance. Then Prescott and Su  \cite{PreSu} extended the study in higher dimensional tori.

The novelty of our approach is to consider estimates as $\varepsilon \to 0$ uniformly in $n$; we show in Section \ref{sec:RW} that for a given $\varepsilon $, irrelevant of the number of steps $n$, there is a probability always smaller than $c_{+}\varepsilon ^{\frac{m}{m+\eta } }$ that the walk on the torus ends up in $B_{d}(0,\varepsilon ) $ after $n$ steps, where $\eta \geqslant 0$ is such that the $\omega _{k},1\leqslant k\leqslant d$ are $q^{-(m+\eta )}$-(BA). The lower bounds obtained in the same section show that these upper bounds are optimal.

\begin{remark}

This value is actually very sensitive to the probability of vanishing coordinates $\overline \U_{n,[k]}$ of $\overline{\U}_{n}$, in the sense that it decays slowly in $\varepsilon $ because of the fast recurrence to $0$ on the axes: for $p<d$
\begin{align*}
\mathbb{P}(\barU_{ n,[1]}=\barU_{ n,[2]}=\dots =\barU_{ n,[p]}=0)\sim n^{-p/2}.
\end{align*}
A heuristic argument is that the  symmetric random walk on $\mathbb{Z} $ has a probability $\sim  n^{-1/2}$ to come back to $0$ in $2n$ steps, and the components are {\it almost} independent up to the parity relation $n\equiv \sum_{k=1}^{d}\U_{n,[k]}$ (see Lemma \ref{lm:Gauss-approx}).

\end{remark}

In the light of the remark above,  only non-vanishing coordinates matter in the speed of decay as $\varepsilon \to 0.$
Denote by $\llbracket d\rrbracket $ the set $\{1,2,\dots ,d\}$. Define for $K\subset \llbracket d\rrbracket ,K\neq \emptyset ,$ the projected ball $B_{K}(\varepsilon ):=B_{d}(0,\varepsilon )\cap H_{K}$ where \begin{align*}
H_{K}:=\{\y=(y_{[k]})_{1\leqslant k\leqslant d}\in \mathbb{R}^{d}:y_{[k]}\neq 0,k\in K\text{\rm{ and }}y_{[k]}=0,k\notin K\}.
\end{align*}

 Then we have according to Theorem \ref{thm:dioph-rw}-(i):
 \begin{corollary}
 For some $c<\infty $, uniformly on $n,\varepsilon ,$
\begin{align*}
\mathbb{P}(\barU_{n}\in B_{K}(\varepsilon ) )\leqslant c n^{-\frac{(d- | K |)m }{2}} \varepsilon ^{\frac{|K|m}{m+\eta }}
\end{align*} 
\end{corollary}
Regarding the dependance in $\varepsilon ,$ the random walk hence comes back to $0$ faster on subspaces with fewer coordinates equal to $0$ (the dependence as $n$ increases is opposite). The most interesting part of the convergence, i.e. where the magnitude is not dominated by coordinates equal to $0$, seems to happen on the domain $H_{\llbracket d\rrbracket }$ of points with non-vanishing coordinates.  More precise results are derived in   \eqref{eq:pn-upper}, in Section  \ref{sec:RW}, dedicated to irrational random walks; the results are derived in particular in terms of the optimal function $\psi $ such that $\bo $'s components are  $\psi $-(BA).

Lower bounds are more unstable, hence we consider the smoothed estimate, for $\beta >2,$
\begin{align}
\label{eq:def-Ibeta} 
\I_{\beta }(\varepsilon )=&\sum_{n\geqslant n_{\varepsilon }}n^{-\beta /2}\mathbb{P}(0< \| \barU_{n} \| <\varepsilon )
\end{align}
where $n_{\varepsilon }\geqslant 1 $ grows sufficiently slowly (see Theorem \ref{thm:dioph-rw}-(ii)). To have matching upper and lower bounds, we assume that for some fixed $ \omega \in \mathbb{R}^{m}$, $ \omega _{[k]}=\omega $ for $ 1\leqslant k\leqslant d,$ and that $\omega $ is  $ \psi $-WA and $\psi $-BA.
We have the following corollary of  Theorem \ref{thm:dioph-rw}-(i),(ii):
\begin{corollary} There are $0<c_{-}<c_{+}<\infty $ such that  $$c_{-}\varepsilon ^{\frac{\beta -2+dm}{m+\eta }}\leqio \I_{\beta }(\varepsilon ) \leqslant c_{+}\varepsilon ^{\frac{\beta -2+dm}{m+\eta }}$$ as $\varepsilon \to 0$.
\end{corollary}

\newcommand{\e}{{\bf e}}

In Section \ref{sec:RW}, similar results (but with different magnitudes in $n$) are actually derived first for the random walk $\U_{n}=\sum_{i=1}^{n}X_{i}$ itself, and  projected on the torus to yield the aforementioned results. 

\subsection{Hyperuniform models}
\label{sec:structure}
\renewcommand{\S}{\mathcal  S}

We have just observed, for some values of the parameter $\bo$, the suppression of the variance at large scales, also called \emph{hyperuniformity phenomenon}. A more general mathematical indicator of hyperuniformity is through  the structure factor, or more generally the behaviour around zero of the Fourier transform of the associated random measure. 
Following \cite{OSST}, we use the integrated structure factor to characterize  hyperuniformity.

\begin{definition}Let $ E$ be a random subset of $ \mathbb{R}^{d}$.  The structure factor of $ E$, when it exists, is the measure $ \mathcal S$ on $ \mathbb{R}^{d}$ defined  through test functions $ \varphi $ smooth with compact support via
\begin{align*}
\int_{\mathbb{R}^{d}}\varphi (\t)\text{\rm{Cov}}(\mathbf{1}_{\{0\in E  \}},\mathbf{1}_{\{\t\in E \}})d\t=\int_{\mathbb{R}^{d}}\hat \varphi (\x)\S (d\x)
\end{align*}
where $ \hat \varphi $ is the classical  Fourier transform of $ \varphi .$
Then say that $ E$ is \emph{hyperuniform} if $ \mathcal S(B_{d}(0,\varepsilon ))\leqslant c_{+}\varepsilon ^{\alpha +d},\varepsilon >0$ for some $ \alpha >0,$ and \emph{strongly hyperuniform} if furthermore $ \alpha>1.$ 
\end{definition}
The precise definition of a random set is not precised here as many existing theories can fit in the previous definition (see for instance \cite{Mol05}).
The hyperuniformity of $ E$ corresponds under some hypotheses to  the suppression of the variance at large scales, i.e. 
\begin{align*}
\lim_{T\to \infty }\frac{ \Leb(E\cap TW)}{T^{d}}\to 0
\end{align*}
for sufficiently regular shapes $ W\subset \mathbb{R}^{d}$,  and the strong hyperuniformity corresponds to a variance of minimal magnitude, proportional to the window boundary measure, i.e. 
\begin{align*}
\sup_{T}\frac{ \Leb(E\cap TW)}{T^{d-1}}<\infty ,
\end{align*} see for instance \cite{OSST}, or the survey in preparation \cite{Coste}.
Our first result is that  natural models of Gaussian fields will not yield hyperuniform excursions.
\begin{proposition}
\label{prop:isotropic}
Let $X$ be some centred stationary Gaussian field on $\mathbb{R}^{d}$ with spectral measure  $\bmu$.
Assume that for some odd integer $n\geqslant 1,\varepsilon >0,x\in \mathbb{R}^{d},c>0,$ $\bmu^{n}(B_{d}(x,\varepsilon ))\geqslant c\varepsilon ^{d}$.
Then for some $c_{-}>0, T$ sufficiently large, $$V_{\bmu}(T)\geqslant c_{-}T^{d}.$$
This is for instance the case if $X$ is isotropic, i.e. if $X$'s law  is invariant under rotations, or equivalently if $\bmu$ is invariant under rotations.
\end{proposition}

The proof requires tools and notation from Section \ref{sec:exc} and is at  Section \ref{sec:isotropic}. As illustrated by the proof, to obtain sublinear variance, the spectral measure's support must have essentially  dimension  smaller than $1$, hence we consider finite atomic support. 
Let $\bmu$ be of the form \eqref{mu:general} with $ \omega _{[k]}\in \mathbb{R}^{m}$ that is $q^{-(m+\eta )}-$(BA) for some $ m\geqslant 1,\eta \geqslant 0,$  for $1\leqslant k \leqslant d$, and $ X_{\bo}$ the Gaussian  field which spectral measure is $\bmu$. 

\begin{proposition}
\label{prop:structure}
 Let $ \alpha =\frac{ 1+d(1-\eta )}{m+\eta }$.
Then $ E=\{X_{\bo}>0\}$ admits a  structure factor $ \mathcal S$ satisfying 
\begin{align*}
\mathcal  S(B_{d}(0,\varepsilon ))\leqslant c_{+}\varepsilon ^{d+\alpha } ,\varepsilon \to 0.
\end{align*}
Hence if $\eta <1+\frac{ 1}{d}$,  $ E$ is hyperuniform, and if $\eta \leqslant 1-\frac{ m}{d+1}$, $ E$ is strongly hyperuniform.

\end{proposition}

 The proof is at Section \ref{sec:prf-structure}.
 If the $ \omega _{[k]}$ are $ \psi $-(SWA) (Simultaneously Well Approximable, see Section \ref{sec:RW}), which is the case if for instance the $ \omega _{[k]}$ are all equal to a $ \psi $-WA tuple $ \omega \in \mathbb{R}^{m}$, the right hand side is optimal, see \eqref{eq:Jb0-lower}.
We give an approximate representation in Figure \ref{fig:1} in a special case. 
\begin{center}
\begin{figure}[h!]
\caption{Structure factor for $d=2,m=1,\omega_{[1]}=\omega _{[2]} =\sqrt{2},$  i.e. $\bmu=\bar \delta _{\be_{1}}+\bar \delta _{\sqrt{2}\be_{1}}+\bar \delta _{\be_{2}}+\bar \delta _{\sqrt{2}\be_{2}}$}
\label{fig:1}
 \includegraphics[scale=.5]{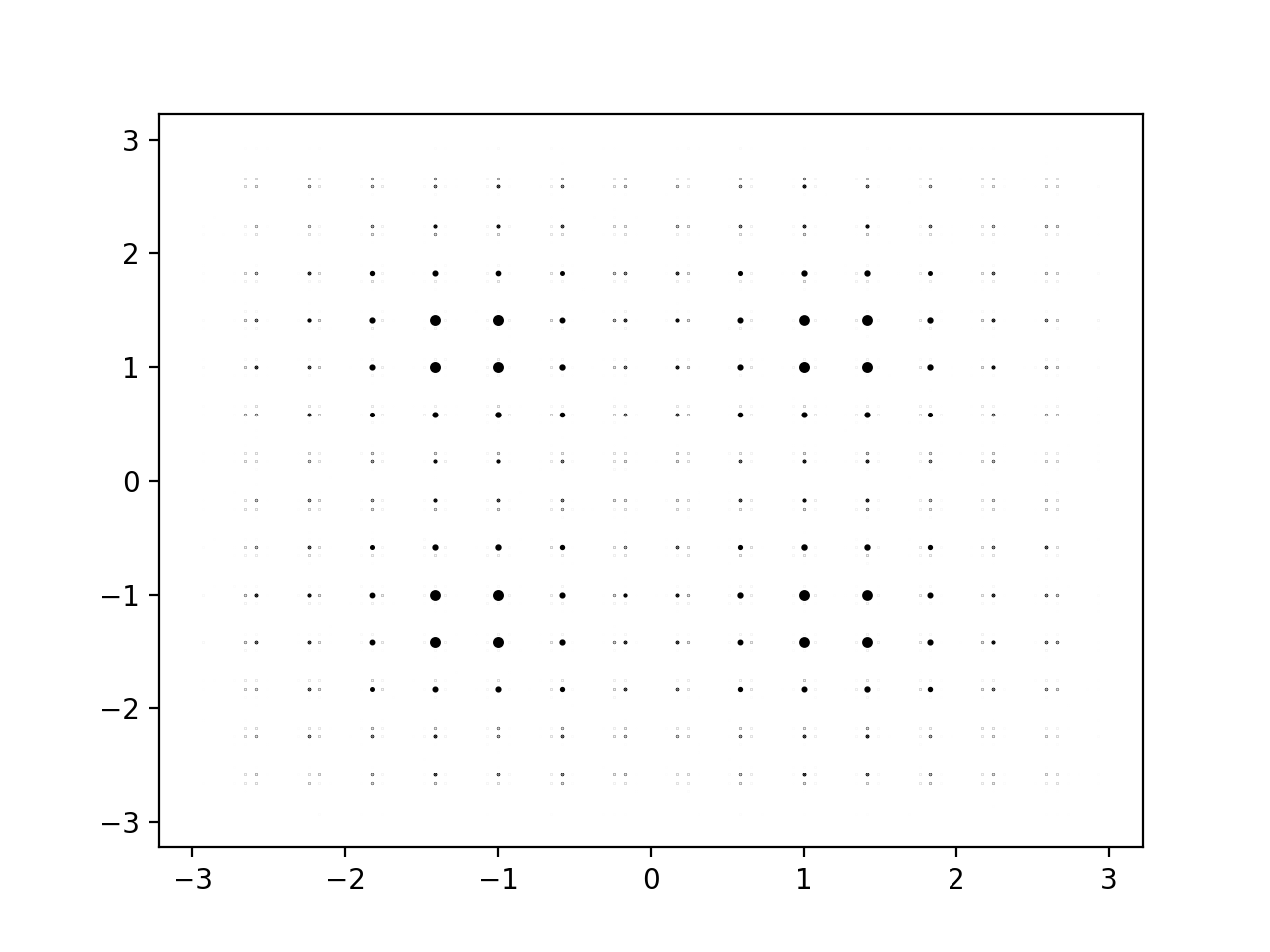} 
\end{figure}
\end{center}
This kind of spectrum is reminiscent of Bragg peaks in quasi-crystals \cite{Tor16}, and more generally of almost periodic fields, for which we give a definition here: a field $X:\mathbb{R}^{d}\to \mathbb{R}$ is {\it almost periodic} if for any sequence of vectors $\t_{n}\to \infty $, there is a subsequence $\t_{n'}$ such that $\|X-X(\t_{n'}+\cdot )\|_{\infty }\to 0$, see for instance the monograph \cite{CG89}. Covariance functions and random Gaussian fields considered in this paper are $\|\cdot \|_{\infty }$-limits of trigonometric polynomials, and as such they are almost periodic. On the other hand, their excursions, seen as $\{0,1\}-$valued functions, are not almost periodic in this sense, mainly because of the discontinuities at the set boundary. On the other hand, they are likely almost periodic for weaker norms, and could hence be seen as {\it almost periodic sets}.

 \subsection{Variance cancellation for Gaussian random waves}
\label{sec:intro-waves}

Let us give another (and unrelated) application of Theorem \ref{thm:general-var} in the context of ergodic isotropic fields.  This example is mainly derived to illustrate the sensibility of Theorem \ref{thm:general-var}, and could likely be deduced from (more precise)   results  on the sphere by Rossi \cite{Ros19}.

Let $d\geqslant2,\bS^{d-1}=\{\t\in \mathbb{R}^{d}:|\t|=1\}$ the  $d$-dimensional unit sphere and $\bmu _{d}$ the Haar distribution on the sphere, i.e. the unique probability measure on $ \mathbb S^{d-1}$ invariant under rotations. Let $X_{d}(\t)$ be a centred Gaussian random field with spectral measure $\bmu _{d}$ and reduced covariance 
\begin{align}
\label{eq:rpw-cov}
C_{d}(\t)=\int_{\bS^{d-1}}\exp(-\imath \t\cdot \x)\bmu _{d}(d\x)=c_{d} \frac{ \mathsf B_{\frac{ d}{2}-1}(  \| \t \| )}{  \| \t \|  ^{\frac{ d}{2}-1}} ,\t\in \mathbb{R}^{d}
\end{align}
for some $c_{d}>0$,
see \cite[(21)]{GikSko}, where $\mathsf B_{\frac{ d}{2}-1}$ is the Bessel function of the first kind (see also Example   \ref{ex:sphere}).

The field $ X_{d}$, called {\it Berry's random wave model}, is of central importance as it is the unique stationary isotropic field satisfying $\Delta X_{d}=-X_{d}$ a.s. \cite{NPR}. 
It can be seen as a local approximation of random  eigenfunctions of the Laplacian on compact $ d-$dimensional smooth manifolds, of high interest in the physics literature.  Since nodal statistics are local quantities, it makes sense to expect analogies between the behaviours of random waves  on different smooth manifolds as $ T\to \infty .$ Such random Gaussian harmonics have been recently heavily studied in dimension $ 2$ in the mathematics literature, especially on the sphere or the torus \cite{KKW13,MarWig,EL16,MPRW}, and the Euclidean version $ X_{d}$ has also been investigated, through its percolation properties \cite{MRV} or the statistical properties of its nodal lines \cite{NPR,EL16}. An interesting feature of Gaussian random harmonics is the variance cancellation phenomenon, i.e. the very small asymptotic fluctuations of some statistics of the excursion set at the level $ u=0$ in the high energy limit, compared to other levels $ u\neq 0$. First conjectured by Berry \cite{Ber02} for the length of the excursion boundary on the torus ({\it nodal lines}), it has then been observed and deeply analysed in several other instances \cite{KKW13,MPRW}.

We prove here that a variance cancellation at the level $ u=0$ also occurs for the nodal excursions of the Euclidean Gaussian random waves in any dimension. More specifically, while the excursion volume is overfluctuating for levels $ u\neq 0$ (i.e. the volume of large windows is negligible with respect to the variance of the excursions of $X_{d}$ restricted to this window), the fluctuations are linear for $ u=0$, as would be the case for  fluctuations of short range random fields such as the Bargmann-Fock field (see Example \ref{ex:BF-field}).  For $ T>0$, the rescaled version $ X_{d,T}( \t)=X_{d}( {T}\t)$ satisfies $ \Delta X_{d,T}=-T^{2}X_{d,T}$ and hence can be compared to random harmonics with same wavelength on compact manifolds.
\begin{theorem}
\label{thm:var-cancel}Denote by $V^{u}(T)$ the variance of $\Leb[d](B_{d}(0,1)\cap \{X_{d,T}>u\})$.
For $u\neq 0$, there is $c_{u}>0$ such that
\begin{align*}
c_uT^{1-d}\leqslant V^{u}(T),T>0
\end{align*}
and there is $0<c_{-}\leqslant c_{+}<\infty $ such that 
\begin{align*}
c_{-}T^{-d}\leqslant V ^{0}(T)\leqslant c_{+}T^{-d}, T>0.
\end{align*}
\end{theorem}

The proof is given in Section \ref{sec:proof-var-cancel}. This result can be compared with similar results on the sphere, see
 the work of Marinucci and Wigman in dimension $ 2$ \cite{MarWig}, and then of Rossi \cite{Ros19} in dimension $ d\geqslant 2$, who study the excursion volume (also called {\it defect volume} after centering) of spherical Gaussian harmonics $X$ satisfying $ \Delta_{\mathbb  S^{2}} X=-\ell( \ell+d-1)X,\ell\in \mathbb{N}$, where $\Delta_{\mathbb  S^{2}} $ is the Laplace-Beltrami operator on the sphere. They also obtain a variance of magnitude $ \ell^{1-d}$ at levels $ u\neq 0$ and $ \ell^{-d}$ at the level $ u=0$, echoing experimental results from Blum, Gnutzmann and Smilansky \cite{BGS}. 
Hence the present results are  consistent with those obtained on the sphere.

\begin{remark}Various cancellation phenomena have been explained by the cancellation of the second order Wiener chaos of the corresponding functional, see the seminal work of Marinucci et al.\cite{MPRW}. The proof of the previous result is based on spectral analysis with the spectral measure $\bmu$ of the Gaussian field, emphasizing the fact that the support dimension for low order powers $\bmu ^{2},\bmu ^{3}$ of the spectral measure is crucial in understanding the variance behaviour, it hence sheds a different light on this variance cancellation phenomenon.

\end{remark}

\section{General variance estimates}
\label{sec:exc}

Let $X$ be some real centred stationary Gaussian field on $\mathbb{R}^{d}$, denote its spectral measure by $\bmu$ and its reduced covariance function  by $C$ (see Section \ref{sec:intro-Gauss}). 
We study here the  statistic
\begin{align*}
M_{u}^{\gamma }=\int_{\mathbb{R}^{d}}\mathbf{1}_{\{X(\t)>u\}}\gamma ( \t)d\t
\end{align*}
where  $u\in \mathbb{R}$ and $\gamma $ is some measurable {\it window function}, bounded with compact support with non-empty interior. Define
\begin{align*}
\hat{\gamma }(\x)=\int_{\mathbb{R}^{d}}e^{-i \t \cdot \x   }\gamma (\t)d\t,\x\in \mathbb{R}^{d},
\end{align*} and $\gamma _{T}(\t)=\gamma (T^{-1}\t),T>0$. 
The  variance is
\begin{align*}
V^{\gamma,u }_{\bmu }(T):=\Var( M_{u}^{\gamma _{T} } ) 
\end{align*}
and we use the shortcut notation $V_{\bmu}^{\gamma }:=V_{\bmu}^{\gamma ,0}$.
The most prominent example is the unit sphere indicator $\gamma^{d} :=\mathbf{1}_{\{B_{d}(0,1)\}}(\cdot )$, and in this case $\gamma^{d} $ is also implicit in the notation $V^{\gamma ^{d},u}_{\bmu}=V_{\bmu}^{u},V^{0}_{\bmu}=V_{\bmu}.$

Introduce the notation $$A=\Theta B$$ for two quantities $A,B$ to mean that $c_{-}A\leqslant B\leqslant c_{+}B$ for some $0< c_{-}\leqslant c_{+}<\infty $, on their domains of definition. Without loss of generality, we use the convention 
\begin{align*}
C(0)=\bmu(\mathbb{R}^{d})=1
\end{align*}
as it allows to adopt the probability formalism and eases certain arguments. Denote by $\U_{n}$ the random walk which increment has law $\bmu$ (i.e. $\U_{n}$'s law is $\bmu^{n}$).
Define the function 
\begin{align*}
\K(\varepsilon )=\sum_{n\in \mathbb{N}\text{\rm{ odd}}}\frac{{n\choose 2n}}{4^{n}(2n+1)}\mathbb{P}( \| \U_{n} \| \leqslant \varepsilon ).
\end{align*}
  Say that $\bmu$ is $\mathbb{Z} -$free if $\mathbb{P}(\U_{2n+1}=0)=0$ for $n\in \mathbb{N}$. For $r>0$, denote by   $c_{r}^{+},c_{r}^{-}$ respectively the supremum and infimum of $ \| \hat\gamma (\x) \|^{2} $ for $\x\in B_{d}(0,4 r).$ Notice that $c_{r}^{-},c_{r}^{+}\to  | \hat \gamma (0) | ^{2}>0$ as $r\to 0.$

\begin{theorem}
\label{thm:general-var}Assume $\bmu$ is $\mathbb{Z} -$free and let $r>0$.
\begin{itemize}
\item [(i)]  For $T>0$
\begin{align}
\label{eq:var-lower}
c_{r}^{-}T^{2d}\K (rT^{-1})\leqslant V_{\bmu }^{\gamma }(T).
\end{align} 
\item [(ii)] If in addition for some $c_{1}<\infty $, $|\hat\gamma (\x)|\leqslant c_{1}  \|\x\| ^{-\frac{d+1}{2}}$ for $\|\x\|>4r ,$
then 
\begin{align}
\label{eq:var-upper}
V_{\bmu }^{\gamma }(T)\leqslant  c_{r}^{+}T^{2d}\K  (rT^{-1})+c_{1}T^{d-1}\int_{0}^{(T/r)^{d+1} }\K (y^{-\frac{1}{d+1}})dy .
\end{align} 
In particular if $\gamma =\gamma ^{d}$ and $\K(\varepsilon )=\Theta \varepsilon ^{\alpha }$ as $\varepsilon \to 0$ for some $\alpha >0$, then $\alpha \leqslant d+1$ and $$V_{\bmu } (T)=\Theta T^{2d-\alpha },T>0.$$
\item[(iii)] For $u\neq 0,T>0$,
\begin{align*}
V_{\bmu}^{\gamma,u}(T)\geqslant 2^{2d} c_{r}^{-}\alpha _{2,u} T^{2d}\mathbb{P}( \| \U_{2} \| <rT^{-1}).
\end{align*}
In particular if $\bmu $ has an atom at $\x_0$,  letting $r\to 0$ yields a quadratic variance
\begin{align*}
0<2^{2d} | \hat\gamma (0) | ^{2}\alpha _{2,u}\bmu(\{\x_0\})^2T^{2d}\leqslant V_{\bmu}^{\gamma,u}(T)\leqslant T^{2d}.
\end{align*}
\end{itemize}
\end{theorem}

 The proof, deferred to Section \ref{sec:prf-main-var}, is based on an exact local formula for the variance of Gaussian excursions volume.
  This expression is then decomposed in two terms through a truncation, the first one is proportional to $T^{2d}\K(rT^{-1})$.
   The second term  is upper bounded in the following way, 
\begin{align*}
\int_{0}^{(T/r)^{d+1} }\K (\lambda (y))dy \leqslant c\int_{0}^{(T/r)^{d+1} }\K (cy^{-\frac{1}{d+1}})dy 
\end{align*}
where $\lambda $ is a pseudo-inverse of $\hat \gamma ^{2}$, giving the second term on the right hand side of \eqref{eq:var-upper}. Since $\hat\gamma $ usually experiences oscillations at $\infty $ (see for instance \eqref{eq:gamma-sphere}), obtaining a simple asymptotic equivalent of the left hand term requires more involved computations, but doing so would provide an accurate lower bound on the variance. 

The second term has the same magnitude than the first term in (at least) three very different settings: (a) when $\bmu$ has finitely many incommensurate  atoms (Theorem \ref{thm:dioph-rw} and Proposition \ref{prop:psi-regular},
   (b) when $\bmu$ is the Haar measure on the unit sphere (Theorem \ref{thm:var-cancel}), and (c) when $\bmu$ is the Fourier transform of an integrable function (Example \ref{ex:BF-field}); see also the end of Section \ref{sec:prf-main-var} for a general argument. 

\begin{remark}The fact that if $\K(\varepsilon )=\Theta \varepsilon ^{\alpha }$ as $\varepsilon \to 0$ for some $\alpha >0$  then $\alpha \leqslant d+1$ (point (ii)), is a  non-trivial fact about random walks on $\mathbb{R}^{d}$ without any assumption on the increment measure, we don't know if this fact is known, or useful, in the study of random walks.
\end{remark}

 \begin{remark}Theorem 2.1 requires almost no hypothesis on the model, except that the random walk does not have a positive probability to come back to $0$ in an odd number of steps ($\bmu $ is $\mathbb{Z} $-free). This is also valid for the result about random walks in the remark above. To have an upper bound in (ii), we still need to ensure that the observation window is smooth enough, through a decay hypothesis on its Fourier transform. 
\end{remark}

\begin{example}
\label{ex:sphere}
For the unit sphere indicator,
we have the classical formula (\cite[Chap. 1.5]{GikSko})
\begin{align*}
\hat\gamma ^{d}(\x)=\kappa _{d}\|\x\|^{-d/2}\mathsf B_{d/2}(\|\x\|)
\end{align*}
where $\kappa _{d}=\Leb[d](B_{d}(0,1))$ and $ \mathsf B_{a}$ is the  Bessel function of the first kind with parameter $a$
\begin{equation*}
\mathsf B_a(r)=\sum_{m=0}^{\infty }\frac{(-1)^{m}}{m!\Gamma (m+a+1)}\left(
\frac{r}{2}
\right)^{2m+a},r\geqslant0.
\end{equation*}
 In particular, $\hat\gamma ^{d}(\x)\sim \kappa _{d}\Gamma (d/2+1)^{-1}>0$ in $0$ and 
\begin{align}
\label{eq:gamma-sphere}
\hat \gamma ^{d}(\x)\sim \kappa _{d}(2/\pi )^{1/2}\|\x\|^{-\frac{d+1}{2}}\cos(\|\x\|+\Delta  _{d})
\end{align}as $\x\to\infty $, for some $\Delta _{d}\in \mathbb{R}.$  It is known \cite{Elbert01} that the first zero of $ \mathsf B_{a},a\geqslant 1/2$ is larger than the first zero of $ \mathsf B_{1/2}$, which  is $\pi $ , hence we can take $r=\frac{1}{2}$ in Theorem \ref{thm:general-var}.
\end{example}

\begin{example}
\label{ex:BF-field}
The most studied   Gaussian fields are probably  those with an integrable reduced covariance function 
\begin{align*}
\int_{ \mathbb{R}^{d}}C(\t)d\t<\infty ,
\end{align*}
such as the Bargmann-Fock field, where $ C(\t)=e^{-\t^{2}}$. Let us emphasise that the following result is far from new, but, along with Theorems \ref{thm:main-intro} and \ref{thm:var-cancel}, it illustrates the variety of situations where Theorem \ref{thm:general-var} is accurate.

\begin{proposition}
\label{prop:integrable}
Let $ \bmu$ be a spectral measure with integrable covariance.
We have $$ V_{\bmu}(T)=\Theta T^{d},T>0.$$
\end{proposition}
\begin{proof}

The integrability of $ C$ yields that $ \bmu$ admits a continuous bounded density function with respect to $ \Leb$, denoted by $ \hat C$, satisfying $\|\hat C\|_{L^{1}}=1$.  Hence, denoting by $\hat C^{*n}$ the $n$-fold self convolution of $\hat C,$
\begin{align*}
\mathbb{P}(\|\U_{n}\|\leqslant \varepsilon )=\int_{}\mathbf{1}_{\{B_{d}(0,\varepsilon )\}}(\z)\bmu ^{n}(d\z)\leqslant \Leb(B_{d}(0,\varepsilon ))\|\hat C^{*n}\|_{\infty }\end{align*}
and classical properties of the convolution product yield that $$ \|\hat C^{*n}\|_{\infty }\leqslant \|\hat C\|_{\infty }\|\hat C\|_{L^{1}}^{n-1}=\|\hat C\|_{\infty }, $$ hence $$\K(\varepsilon )\leqslant \|\hat C\|_{\infty }\left(
\sum_{n}\frac{{n\choose 2n}}{4^{n}(2n+1)}
\right)\varepsilon ^{d}.$$ Theorem \ref{thm:general-var}-(ii) then gives the upper bound.

For the lower bound, recall that $\hat C$ is continuous and semi definite positive, and let $ \x\in \mathbb{R}^{d},r >0$  be such that $ \hat C>0$ on $ B_{d}(\x,r)$. It is then easy to show by induction that $ \hat C^{*n}>0$ on $ B_{d}(\x,nr)$, hence for $ n\geqslant r^{-1}(|\x|+1 )$ and $\varepsilon \leqslant 1$, $ B_{d}(0,\varepsilon )\subset B_{d}(\x,nr)$ and 
\begin{align*}
\mathbb{P}( \| \U_{n} \| \leqslant \varepsilon )=\int_{ B_{d}(0,\varepsilon )}C^{*n}(\t)d\t \geqslant  c'\varepsilon ^{d}
\end{align*} for some $ c'>0$, Theorem  \ref{thm:general-var}-(i) concludes the proof.
\end{proof}

\end{example}

  \section{Irrational  random walks}
  \label{sec:RW}
  
  We consider a random walk in $\mathbb{R}^{d}$ which increment measure $\bmu$ is symmetric with finite support. For technical reasons, it is simpler to assume that $\bmu$ has atoms along some $d$ linearly independent unit vectors $\be_{1},\dots ,\be_{d}$, with the same number of atoms in each direction: for some $m\geqslant 1$, let $\bo =(\omega _{[k],i})_{\substack{\text{$1\leqslant k\leqslant d$}\\\text{$1\leqslant i\leqslant m$}}}\in (\mathbb{R}^{m})^{d}$ and $ \bmu$ be of the form \eqref{mu:general} with $ \omega _{[k],0}=1$ by convention.
We are interested in the associated random walk 
\begin{align*}
\U_{n}:=\sum_{k=1}^{n}X_{k} 
\end{align*}
where the $X_{k}$ are independent and identically distributed with law $\bmu $, hence centred.
The study of $\U_{n}$ is related to the random walk on the torus
\begin{align*}
\overline{\U}_{n}=\U_{n}-[\U_{n}]\in [0,1[^{d},
\end{align*}
which has been intensively studied, the consequences of the current results to the random walk on the torus are discussed at  Section \ref{sec:intro-walk}.  
To avoid degenerate behaviour, we assume that $\bmu$ is   $\mathbb{Z} $-free, i.e. there is no odd $\q \in \mathbb{Z} ^{M}\setminus \{0\}$ such that $\sum_{i=1}^{M}\q_{i}\bo_{i}=0$, where $ M=(m+1)d$. In general we further assume that the $\omega _{[k]}$ are $\psi $-BA for some   non-vanishing function $\psi $, which automatically implies that $\bmu$ is $\mathbb{Z} $-free.

\newcommand{\N}{{\bf N}}

According to the Central Limit Theorem, the law of the renormalised sum $n^{-1/2}\U_{n}$ weakly converges to a Gaussian measure (see also Lemma \ref{lm:Gauss-approx} for  precise estimates), and the law $\overline{\bmu _{n}}$ of $\overline{\U}_{n}$ is known to converge to Lebesgue measure on $[0,1[^{d}$ \cite{Su98}. But  if we zoom in further on this convergence around $0$, it becomes very erratic.
\renewcommand{\x}{{\bf x}}
 We estimate the following quantities below:\begin{align*}
p_{n}^{\x }(\varepsilon )=&\mathbb{P}(0< \| \U _{n}-\x \| \leqslant \varepsilon ),\varepsilon > 0,\x\in \mathbb{Z}^{d} ,\\
\bar p_{n}(\varepsilon )=&\sum_{\x\in \mathbb{Z}^{d} }p_{n}^{\x}(\varepsilon )=\mathbb{P}(0< \| \overline{\U}_{n} \| \leqslant \varepsilon ).
\end{align*}

\begin{remark}
In general, if the sum $n+\sum_{k=1}^{d}\x_{[k]}$ is even, $\mathbb{P}(\U_{n}=\x )$ is in $n^{-\frac{M}{2}}$ and dominates $p_{n}^{\x}(\varepsilon )$ for $\varepsilon \to 0$, which is why it is estimated  separately. For odd values, since $\bmu$ is $\mathbb{Z} $-free, $\mathbb{P}(\U_{n}=\x)=0$, hence $p_{n}^{\x}(\varepsilon )$ is simply $\mathbb{P}(\|\U_{n}-\x\|\leqslant \varepsilon )$.
A fine analysis of the recurrence around $0$ yields that the  rate strongly depends on the number of coordinates equal to $0$, expressed through  $$p_{n}^{K,\x}(\varepsilon )=\mathbb{P}(\U_{n}-\x\in B_{K}(0,\varepsilon )),\;\bar p_{n}^{K}(\varepsilon )=\mathbb{P}(\barU_{n}\in B_{K}( 0,\varepsilon )).$$

We show below for instance that if along each direction $ k$ of $\mathbb{R}^{d}$, $\bmu$'s support is made up of  a vector $( \omega _{[k],i})_{1\leqslant i\leqslant m}$ which  is $q^{-(m+\eta )}$-(BA) for some $m\geqslant 1$, $\eta \geqslant 0$, then  for some $c<\infty ,$ 
\begin{align*}
\bar p_{n}^{K}(\varepsilon )\leqslant c  n^{-m\frac{d-|K|}{2}}\varepsilon ^{ \frac{m| K| }{m+\eta }}, \;\;\;K \subset \llbracket d\rrbracket ,n\in \mathbb{N},0<\varepsilon <\frac{ 1}{2}, 
\end{align*}
so that it is really the number of vanishing coordinates that determines the recurrence probabilities. 
\end{remark}

To avoid the technicality mentioned in the previous remark and  obtain lower bounds, we consider the smoothed estimates for $\beta >0,$ with $p_{n}^{\o}=p_{n}^{\llbracket d\rrbracket ,\o}$
\begin{align*}
\J_{\beta }(\varepsilon  ):=&\sum_{n\geqslant n_{\varepsilon },n\in \mathbb{N},n\text{\rm{ odd}} }n^{-\beta/2 }p_{n}^{\o }(\varepsilon )
\end{align*}
where $n_{\varepsilon }$  does not grow too fast as $\varepsilon \to 0$ and serves the purpose to show that  it is the series  tail that actually matters.

 \begin{remark}
 Considering this statistic also allows to suppress the erratic behaviour in $n$, and we can prove that $\J_{\beta }(\varepsilon )$ and $\I_{\beta }(\varepsilon )$ (defined in the introduction at \eqref{eq:def-Ibeta})  both behave in $\varepsilon ^{ \frac{m}{m+\eta } }$, and find a matching lower bound.
The summation over odd $n$  in $\J_{\beta }(\varepsilon )$ is adapted to   estimating the volume variance of Gaussian nodal excursions (see Remark \ref{rk:psi-SWA}).  
\end{remark}
If the atoms are different in different directions, we need to generalise the concept of $ \psi $-WA: 
say that $\bo=(\omega _{[k]})_{1\leqslant k\leqslant d}$ is $\psi $-SWA*  if for some $c>0$,  for infinitely many $q^j \in \mathbb{Z} ^{m}  ,j\geqslant 1,$ there exist $p^{j}_{[k]}\in \mathbb{Z},1\leqslant k\leqslant d $ such that 
\begin{align*}
 | p^{j}_{[k]}-\omega _{[k]} \cdot q^j | <c\psi ( | q^j  | ),1\leqslant k\leqslant d.
\end{align*}Say that $ \bo$ is $ \psi $-SWA 
 if furthermore $
\sum_{k=1}^{d}(p^{j}_{[k]}+\sum_{i=1}^{m} q^{j}_{i})$ is odd. The need to distinguish between $\psi $-SWA and $\psi $-SWA*   is discussed in Remark \ref{rk:psi-SWA}.

\begin{theorem}
\label{thm:dioph-rw}
Let $\psi $ be  some mapping $\mathbb{N}_{*}\to (0,1]$ converging to $0$, and let $\psi ^{-1}$ be its pseudo-inverse defined by 
\begin{align}
\label{eq:pseudo-inv}
\psi ^{-1}(\varepsilon )=\min\{q\in \mathbb{N}^{*}: \psi (q )\leqslant \varepsilon \},\varepsilon >0.
\end{align} 
Let $ \beta >0 $. There is  $0<c<\infty $ depending on $d,m,\psi ,\beta $ such that the following holds:\\
(i)
Assume each $\omega _{[k]}$ is $\psi  $-BA. We have for $\x\in \mathbb{Z} ^{d},K\subset \llbracket d\rrbracket,0<\varepsilon < \frac{1}{2},n\in \mathbb{N}^{*}$
\begin{align}
\label{eq:pn0-upper}p_{n}^{K,\x}(\varepsilon ) \leqslant &c  n^{-d/2}n^{-\frac{(d-|K|)m}{2}}\psi^{-1}(\varepsilon )^{ -m |K|}\exp(-c n^{-1}\|\x\|^{2})\\
\label{eq:pn-upper} \bar p^{K}_{n}(\varepsilon ) \leqslant &cn^{\frac{-(d-|K|)m}{2}}\psi^{-1}(\varepsilon )^{- m |K|}\\
\label{eq:Ib-upper} \I_{\beta }(\varepsilon )\leqslant& c\psi^{-1}(\varepsilon )^{- \beta -d m+2  }\\
\label{eq:Ib0-upper} \J_{\beta }(\varepsilon )\leqslant& c\psi^{-1}(\varepsilon )^{ -\beta -d(m +1)+2 }\end{align}

(ii) Assume $\bo $ is $\psi $-SWA*. Then if $n_{\varepsilon }\leqslant \psi^{-1}(\varepsilon )^{2},\varepsilon >0,$
\begin{align}
\label{eq:Ib-lower} \psi^{-1}(\varepsilon )^{-\beta -dm+2 }\leqio c\I_{\beta }(\varepsilon ),\varepsilon \to 0.
\end{align}

(iii) Assume $\bo$ is $\psi  $-SWA . Then if $n_{\varepsilon }\leqslant \psi^{-1}(\varepsilon )^{2},\varepsilon >0,$
\begin{align}
\label{eq:Jb0-lower} \psi^{-1}(\varepsilon )^{-\beta -d(m+1)+2 }\leqio  c\J_{ \beta }(\varepsilon ),\varepsilon \to 0.
\end{align}

\end{theorem}

This result is useful for studying Gaussian excursions, but as explained at Section \ref{sec:intro-walk}, it independently yields new estimates related to random walks on the torus.

\begin{example}
\label{ex:standard}
An example intensively used in this article is $\beta =3, \psi (q)=cq^{-(m+\eta ) },\eta \geqslant 0, \varepsilon =T^{-1}$ for some large $T$, for which 
\begin{align}
\label{eq:psi-1-power}
\psi ^{-1}(T^{-1} )^{-\beta -d(m+1)+2}=c'T ^{-\frac{ 1+d(m+1)}{m+\eta }},\varepsilon >0.
\end{align}
\end{example}

\begin{remark}
\label{rk:psi-SWA}The assumption that $\bo$ is $\psi $-SWA is stronger than $\psi $-SWA*, and also less natural, which might cause confusions. The reason why $\bo$ has to be $\psi -$SWA  instead of   $\psi -$SWA* at point (iii) is because summands are odd in the definition of $\J_{\beta }$. The assumption $\psi $-SWA*  and function $\I_{\beta }$ are introduced only because they are more natural in the context of random walks, but they are useless for giving lower bounds for Gaussian excursions variances.  As supported by Proposition \ref{prop:psi-SWA*}, this subtlety does not influence final results about Gaussian excursions, hence one would like a general result from diophantine approximation that states that $\psi -$SWA* tuples are also $\psi -$SWA, but that is most likely not true.
\end{remark}

In the next section, results of the two previous sections are combined to yield estimates for the volume variance of Gaussian random fields which spectral measure is made of incommensurate atoms. One can also use the next section to have example behaviours of irrational random walks.

\section{Variance asymptotics for diophantine Gaussian excursions}
\label{sec:intro-Gauss-dioph}

 We consider   symmetric spectral measures whose support contains incommensurate atoms. For $\bo\in (\mathbb{R}^{m})^{d}$, denote by $X_{\bo}$ a centered stationary Gaussian random field which spectral measure is $\bmu $, parametrised by $\bo$ as  in  \eqref{mu:general}.
The excursion volume variance is denoted by
\begin{align*}
V_{\bo }(T):=\Var(\Leb(\{X_{\bo }>0\}\cap B_{d}(0,T))).
\end{align*}

\subsection{Regular asymptotics}

 We will use the assumptions that $\bo$ is $\psi $-BA and / or $\psi $-WA (Section \ref{sec:intro-Gauss}) with functions $ \psi $ of the following form : \begin{definition}
 \label{def:regular}
 Say that $\psi:\mathbb{N}^{*}\to (0,1]$ is \emph{absolutely regularly varying} (ARV) (or $\tau $-ARV) if it is   of the form $\psi (q)=q^{-\tau }L(q),q\in \mathbb{N}^{*},$ where $\tau >0$   and $L$ is \emph{ absolutely slowly varying} (ASV), i.e. it does not vanish and as $q\to \infty ,$ 
\begin{align}
\label{eq:srv}\frac{ L(q+1)}{L(q)}=1+o(q^{-1}).
\end{align}
\end{definition}
 In the classical Karamata Theory \cite[(1.2.1)]{RegVar}, a function $f:[a,\infty [$ is {\it slowly varying} (SV) if $f(qx)/f(x)$ converges to $1$ as $x\to \infty $ for all $q\in \mathbb{R}$. We extend naturally this definition to a function $L:\mathbb{N}\to \mathbb{R}$ by  
\begin{align}
\label{eq:sv}
\frac{ L([aq])}{L(q)}\xrightarrow[q\to \infty ]{}\;1, \text{ for all }a>0,
\end{align}
where $[aq]$ is the integer part of $aq.$
The terminology  ASV is motivated by  the following result, proved at Section \ref{sec:rv}, studying the relation with SV.

\begin{proposition}
\label{prop:rv}
If $L$ is ASV, then $L$ is SV. The converse is true if $L$ is non-decreasing or non-increasing, otherwise it might not hold.
\end{proposition}

 We  introduce the pseudo-inverse $\psi ^{-1}:(0,1]\mapsto \mathbb{N}$ by \eqref{eq:pseudo-inv}.  We can show that if $\psi $ is ARV, for every  finite $r>0$ there are finite $c_{i}>0$ such that  $c_{1}\psi (q)\leqslant \psi (rq)\leqslant c_{2}\psi (q)$ and $c_{3}\psi ^{-1}(\varepsilon )\leqslant \psi ^{-1}(r\varepsilon )\leqslant c_{4}\psi ^{-1}(\varepsilon )$ on their domains of definition.
Remark that $q\psi (q)$ is strictly non-increasing for sufficiently large $q$ if $\tau>1$. 
 
Our most precise and general result concerns the case where the frequencies $\omega _{[k]}$ of $\bmu$ are the same in all $d$ directions, i.e. $\omega _{[k],i}=\omega _{[1],i},$ it yields stationary random sets  in $\mathbb{R}^{d}$  with any reasonable asymptotic prescribed variance behaviour.

\begin{theorem} \label{thm:main-intro}
Let $\tau >0,\psi$ $\tau $-(ARV)  and $\omega \in \mathbb{R} ^{m}$ that is $\psi $-WA  and $\psi $-BA, and  $\bo:=(\omega ,\omega ,\dots ,\omega )\in  (\mathbb{R}^{m})^{d}$. Then with $\tau ^{*}=\frac{1+d(m+1)}{1+d}$, as $T\to \infty ,$ 
\begin{align*}
\frac{c_{-}T^{2d}}{\psi ^{-1}(T^{-1})^{1+d(m+1)}}\leqio V_{\bo}(T)\leqslant &\frac{c_{+}T^{2d}}{\psi ^{-1}(T^{-1})^{1+d(m+1)}}=o(T^{2d})&\text{\rm{ if }}&\tau >\tau ^{*}\\
c_{-}T^{d-1}\leqio V_{\bo}(T)\leqslant &c_{+}T^{d-1}\ln(T)&\text{\rm{ if }}&\psi (q)=q^{-\tau ^{*}}\\
c_{-}\frac{ T^{d-1}}{\ln(T)^{\alpha }}\leqslant V_{\bo}(T)\leqslant& c_{+}T^{d-1}&\text{\rm{ if }}&\psi (q)\geqslant q^{-\tau ^{*}}\ln(q)^{\frac{ 1}{d}}
\end{align*}
for some $0<c_{-}\leqslant c_{+}<\infty $ depending on $d,m,\psi $, and $\alpha =\frac{ 1+d(m+1)}{d\tau ^{*}}$. If $m=1$ and $\tau >1$, there are uncountably many $\omega \in \mathbb{R}$ satisfying the assumption.
\end{theorem}

The proof is in Section \ref{sec:proof-main-psi}.
For $ m\geqslant 2$, most studies  concern power functions $ \psi (q)=q^{-\tau }$ and are discussed in Proposition \ref{cor:intro}, but can likely be extended to more general functions $ \psi $.

\begin{remark}
 
The presence of the term $T^{d-1}$ on the right hand side, proportional to the surface measure of the observation window, is natural as  random stationary measures applied to a large window are usually not expected to have  a variance behaviour lower than the boundary measure.  No rigorous general result in this direction is known by the author, Beck \cite{Beck87} gives a formal proof in the case of point processes. See also  \cite{Tor18}, which classifies hyperuniform behaviours in three types: type I have asymptotic variance in $T^{d-1}$, type II in $T^{d-1}\ln(T)$, and type III gathers all other sublinear behaviours, which actually correspond to the three cases above. Intermediate behaviours between $T^{d-1}$ and $T^{d-1}\ln(T)$ can likely be obtained by the same method.

It is likely that the upper bound in $T^{d-1}\ln(T)$ is sharp if $\psi (q)=q^{-\tau ^{*}}$, proving it rigorously would require a lower bound for $v_{T}^{(2)}$ in the proof of Theorem  \ref{thm:general-var} at Section \ref{sec:prf-main-var}, which raises some technical difficulties because of the cosine term.

\end{remark}

\begin{remark}
This type of behaviour is really specific of nodal excursions. The volume variance for excursions $\{X>u\}\cap B_{d}(0,T)$ always behaves in $T^{2d}$ if $u\neq 0$, see Theorem \ref{thm:general-var}-(iii). The phenomenon of variance cancellation at $u=0$ is heavily documented for Gaussian random waves (see Section \ref{sec:intro-waves}).
\end{remark}

\subsection{Power functions}
 
Prominent examples are provided by power functions, for which we introduce a special notation: say that $\omega\in \mathbb{R}^{m} $ is $(\tau) $-BA (resp. $(\tau) $-WA) if it is $cq^{-\tau }$-BA (resp. $cq^{-\tau }$-WA) for some finite $c>0$.
These considerations are further developed and commented in the Appendix \ref{sec:dioph}{, let us simply mention that for $\eta >0$, $\Leb[m]$-a.a. $\omega \in \mathbb{R}^{m}$ is $(m+\eta )$-BA and   $(m)$-WA.  } For any $\eta \geqslant 0$, there are uncountably many $\omega \in \mathbb{R}^{m}$ that are $(m+\eta )$-BA and $(m+\eta )-$WA.  There are also uncountably many {\it Liouville numbers}, i.e. $\omega \in \mathbb{R}$ that are $(\tau )-$WA for any $\tau >0.$


The following corollary examines different regimes, depending on the relation between $\tau ,d$ and $m$, it is  a consequence of  Theorem \ref{thm:main-intro} for $\psi (q)=q^{-\tau }$.  

\begin{proposition}
\label{cor:intro}
For $\omega \in \mathbb{R}^{m},\tau >0$, let $\bo=(\omega ,\dots ,\omega )\in (\mathbb{R}^{m})^{d},\tau ^{*}=\frac{1+d(m+1)}{1+d}$.
\begin{itemize}
\item[(i)] If $d\geqslant m$, for $\tau \in (m,\tau ^{*})$, $\Leb[m]$-a.a. $\omega \in \mathbb{R}^{m}$ is $ (\tau  )$-BA, and  for some $c_{+}<\infty $  $$V_{\bo}(T)\leqslant c_{+}T^{d-1}, T>0.$$
\item[(ii)] If $d<m$, since for $\Leb[m]$-a.a. $\omega \in \mathbb{R}^{m}$, $\omega $ is $(m)$-WA and $(m+\eta )-$BA for $\eta >0$,  we have $$ c_{-} T^{d- \frac{1+d}{m}}\leqio  V_{\bo}(T)\leqslant c_{+}T^{d-\frac{1+d}{m+\eta }}$$ for some $0<c_{-}, c_{+}<\infty $.
\item[(iii)] Let $m=1.$ For $d-1\leqslant \beta <2d,$ the set of  $\omega \in \mathbb{R}$ such that for some $0<c_{-}\leqslant c_{+}<\infty $
\begin{align*}
c_{-}T^{\beta }\leqio V_{\bo}(T)\leqslant c_{+}T^{\beta },T\geqslant 1
\end{align*}
 is uncountable
(there is  $\tau \geqslant 1$ such that  $\beta =2d-\frac{1+2d}{\tau  }$, and  uncountably many  $\omega $ are $ (\tau  )$-WA and $ (\tau  )$-BA ). 
\item[(iv)] Let $m=1$. For $\omega $ a \emph{ Liouville number}, for every $\varepsilon >0$,  for some $c_{-}>0,$
\begin{align*}
c_{-}T^{2d-\varepsilon }\leqio V_{\bo}(T).
\end{align*}
\item[(v)] In all cases, $V_{\bo}(T)=o(T^{2d})$.
\end{itemize}
All constants $c_{-},c_{+}$ involved only depend on $d,m,\tau .$
\end{proposition}

\begin{proof}For all the proof, recall that $\K=\Theta \J_{3}.$\begin{itemize}
\item [(i)] Let $\omega $ that is $\tau $-BA for some $\tau <\tau ^{*}$. According to  Proposition \ref{prop:psi-regular}, 
\begin{align*}
\max\left\{T^{2d}\J_{3}(T^{-1}) , \quad
T^{d-1}\int_{0}^{T^{d+1} }\J_{3}(y^{-\frac{1}{d+1}})dy\right\}\leqslant cT^{d-1}.
\end{align*}
According to Theorem \ref{thm:general-var}-(ii), it yields the result.
\item[(ii)] Let $\omega $ that is $(m)$-WA and $(m+\eta )$-BA. In particular $\bo$ is $q^{-m}$-SWA by  Proposition \ref{prop:bold-omega-SWA}. Let  $\psi (q)=q^{-m}$. According to Theorem \ref{thm:general-var}-(i),
\begin{align*}
V_{\bo}(T)\geqslant cT^{2d}\K(rT^{-1})
\end{align*} and according to Theorem \ref{thm:dioph-rw}-(iii) (with $n_{\varepsilon }=1$),
\begin{align*}
 cT^{2d}\K(rT^{-1})=\Theta T^{2d}\J_{3}(T^{-1})\geqslant cT^{2d-\frac{ 1+d(m+1)}{m}}=cT^{d-\frac{ 1+d}{m}}.
\end{align*}
For the upper bound, reason like at (i) with $\psi (m)=q^{-(m+\eta )}$. The case $\tau =m+\eta >\tau ^{*}$ applies in Proposition \ref{prop:psi-regular}.

\item[(iii)] This is  Theorem \ref{thm:main-intro} in the case $\tau >\tau ^{*}.$

\item[(iv)] Since $\omega $ is $\tau $-WA with $\varepsilon =\frac{ 1+2d}{\tau }$, $\bo$ is $q^{-\tau }$-SWA. As in (ii), apply Theorem \ref{thm:general-var}-(i) and Theorem \ref{thm:dioph-rw}-(iii) to have the result.

\item[(v)] This is a consequence of Theorem \ref{thm:general-var}-(ii) and the fact that $\K$ is uniformly bounded and converges to $0$  as $\varepsilon \to 0$.
\end{itemize}
\end{proof}

\begin{remark}An interesting observation in dimension $ d=1$ is that in (i), the variance of the excursion indicator is bounded, while its derivative in the distributional sense, i.e. the number of zeros, has maximal quadratic variance, in $ T^{2}$ (see \cite[Theorem 2-(iii)]{Lac20}).
\end{remark}

If  the $\omega _{[k]}$ differ along the directions $1\leqslant k\leqslant d$, 
it is the worst approximable $\omega _{[k]}$ that drives the upper bound, or in other terms the largest function $\psi $ such that each $\omega _{[k]}$ is $\psi $-BA.
 
 \begin{corollary}
 \label{cor:multi}
 Let $\psi :\mathbb{N}_{*}\to (0,1]$ ARV.
 Assume $\bo\in (\mathbb{R}^{m})^{d}$  is such that each $\omega _{[k]}$ is $\psi $-BA, $1\leqslant k\leqslant d$. Then the same upper bounds as in Theorem \ref{thm:main-intro} hold. In particular
 \begin{itemize}
\item [(i)]if $d\geqslant m$, for $\Leb[md]$-a.a. $\bo\in (\mathbb{R}^{m})^{d}$, $V_{\bo}(T)\leqslant c_{+}T^{d-1}$ for some $c_{+}<\infty .$
\item[(ii)] For every $\bo\in (\mathbb{R}^{m})^{d}$, $V_{\bo}(T)=o(T^{2d})$
\end{itemize}
\end{corollary}

 \begin{proof}[Proof of Corollary \ref{cor:multi}] \begin{itemize}
\item [(i)]  

This is essentially the same proof as Proposition \ref{cor:intro}-(i), one simply has to notice that since for any $\tau \in (m,\tau ^{*})$, $\Leb[m]$-a.a. $\omega $ is $(\tau )$-WA,  for any $\tau \in (m,\tau ^{*})$, for $\Leb[md]$-a.a. $\bo=(w_{[k]})_{1\leqslant k\leqslant d}$, each $\omega _{[k]}$ is $(\tau )$-BA, and Proposition \ref{prop:psi-regular} and Theorem \ref{thm:general-var}-(ii) can be applied in the same way.

\item[(ii)] Same proof as Proposition \ref{cor:intro}-(v).
\end{itemize}
\end{proof}

The lower bound really requires simultaneous approximability of the frequencies.
 Theorem \ref{thm:dioph-rw}-(iii) and Theorem \ref{thm:general-var}-(i) yield:
\begin{corollary}Assume that for some function $ \psi:\mathbb{N}^{*}\to (0,1]$ converging to $0$, $\bo$ is  $\psi$-SWA.  Then for some $c_{-}>0 $
\begin{align*} 
 c_{-}T^{2d}\psi ^{-1}(T^{-1})^{-1-d(m+1)}
\leqio  V_{\bo}(T)
\end{align*}
where $\psi ^{-1}$ denotes the pseudo-inverse of $\psi $ (see \eqref{eq:pseudo-inv}).
\end{corollary}
Thanks to Groshev's theorem (see the Appendix \ref{sec:dioph}), for $\eta >0$, $\Leb[md]$-a.a.  $\bo\in (\mathbb{R}^{m})^{d}$ is  $ | q | ^{-m/d}$-SWA but  not $ | q | ^{-m/d-\eta }$-SWA.

\subsection{A randomised model}
\label{sec:randomised}

We build randomised models that exploit the metric results of diophantine approximation to yield  hyperuniform models that are more stable, i.e. not subject to subtle diophantine properties of the parameters.

\begin{proposition}
Let $\Omega $ be a real random variable which law is continuous with respect to Lebesgue measure, and let $a^{i}_{k},i\geqslant 0,k\geqslant 1$ be independent and identically distributed standard Gaussian variables. Define for $\t=(t_{[k]})\in \mathbb{R}^{d},$
\begin{align*}
X(\t)=&\frac{ 1}{2d}\sum_{k=1}^{d}(a^{0}_{k}\cos(t_{[k]})+a^{1}_{k}\sin(t_{[k]})+a^{2}_{k}\cos(\Omega t_{[k]})+a^{3}_{k}\sin(\Omega t_{k})),\\
M_{T}= &\Leb(\{X>0\}\cap B_{d}(0,T))
\end{align*} and $V(T)=\Var(M_{T})$. Then for some $c_{+}<\infty $,
\begin{align*}
V{(T)}\leqslant c_{+}T^{d-1}.
\end{align*} 
\end{proposition}

\begin{proof} 
Since the Gaussian field is centered, for any fixed $\omega \in \mathbb{R}$, 
\begin{align*}
\mathbb{E}(M_{T}\;|\;\Omega =\omega )=\Leb(B_{d}(0,T))/2
\end{align*}
is deterministic. We also know that a.a. $\omega \in \mathbb{R}$ is $(2)$-BA, and if we condition by $\Omega =\omega $, $X$ is the Gaussian field with reduced covariance $\frac{ 1}{2d}\sum_{k=1}^{d}(\cos(t_{[k]})+\cos(\omega t_{[k]}))$. Hence 
 the conditional variance formula  and Proposition \ref{cor:intro} yield
\begin{align*}
V(T)=&\mathbb{E}(\Var(M_{T}\;|\;\Omega ))+\Var(\mathbb{E}(M_{T}\;|\;\Omega ))\\
 &\leqslant c_{+} T^{d-1},
\end{align*}
we emphasize that $c_{+}$ depends only on $d,m,\tau=2 $, and not further on $\omega .$
\end{proof}

The same arguments with $1\leqslant m<d$ yield the following:
\begin{proposition}

Let $(\Omega _{0},\dots ,\Omega _{m})$ a random $(m+1)-$tuple of vectors with continuous law with respect to $\Leb[(m+1)d]$, and  
\begin{align*}
X(\t)=\frac{ 1}{d(m+1)}\sum_{k=1}^{d}\sum_{i=0}^{m}(a_{k}^{2i}\cos(\Omega _{i}t_{k})+a_{k}^{2i+1}\sin(\Omega _{i}t_{k})).
\end{align*}
Then the variance is bounded by $c_{+}T^{d-1}$ if $d>m$.

\end{proposition}

Along similar lines, exploiting Proposition \ref{cor:intro}-(iii) with $m>d$ yields randomised models which variance is in $T^{\beta }$ for some $d-1<\beta <2d$.

\begin{remark}
Similar models in the context of random walks (Section \ref{sec:RW}) yield interesting examples of random walks in a random environment.
\end{remark}

\section*{Acknowledgements}
I am indebted to M. A. Klatt, with whom I had many discussions about hyperuniformity and Gaussian excursions. I also wish to thank S. Torquato for comments on the final manuscript.
I am particularly grateful to Yann Bugeaud for insights about diophantine approximation. 

This study was supported by the IdEx Universit\'e de Paris, ANR-18-IDEX-0001.

   \section{Appendix}

\addtocontents{toc}{\setcounter{tocdepth}{-10}}

\subsection{Diophantine approximation}
\label{sec:dioph}

The core of the paper is provided by results from diophantine approximation, we explain here basic principles and results, as well as the more advanced ones we will need.  The quality of the approximation of a tuple $\omega \in \mathbb{R}^{m}$ is measured by the numbers
\begin{align*}
\d_{q}(\omega )=\inf_{p\in \mathbb{Z} } | p-q\cdot \omega  | ,q\in \mathbb{Z} ^{m}.
\end{align*}
 Given $\psi :\mathbb{N}^{*}\to [0,1]$, the definitions of $\psi -$BA, $\psi -$WA , $\psi $-SWA*, $\psi -$SWA based on this distance are given in the introduction and we complete this picture with the following definition: $\omega \in \mathbb{R}^{m}$ is $\psi $-WA* if for some $c_{\omega }<\infty $, for infinitely many $p\in \mathbb{Z} ,q\in \mathbb{Z} ^{m}$,
\begin{align*}
d_{q}(\omega )\leqslant c_{\omega }\psi (q).
\end{align*}
Proposition \ref{prop:psi-SWA*}, at the end of this section, yields that most quantitative statements available in the literature about $\psi $-WA* tuples also hold for $\psi $-WA tuples.
The most basic, yet useful result is the Dirichlet principle:
 \begin{proposition}Let $m\geqslant 1$. There is $c_{m}<\infty $ such that for $N\in \mathbb{N}^*,$ one can find $q,q'\in B_{N}:=(\mathbb{Z} \cap [-N,N])^{m}$ distinct such that for $\omega \in \mathbb{R}^{m}$, 
\begin{align*}
\d_{q-q'}(\omega )\leqslant N^{-m}\leqslant c_{m}\|q-q'\|^{-m},
\end{align*}
which yields that $\omega $ is $(m)$-WA* and if $\omega $ is $(m+\eta ) $-BA, then necessarily $\eta \geqslant 0.$
\end{proposition}

\begin{proof}Simply remark that if one divides $[0,1]$ in $M :=|B_{N}|-1$ bins of size $M ^{-1}$, out of the $|B_{N}|$ values $ \d_{q }(\omega ),q\in B_{N}$, at least two of them will end up in the same bin, yielding for some $q,q'\in B_{N}$ distinct
\begin{align*}
\d_{q-q'}(\omega ) \leqslant |\d_{q}(\omega )-\d_{q'}(\omega )|\leqslant M^{-1}\leqslant N^{-m}.
\end{align*}
The second inequality comes from $|q-q'| \leqslant 2\sqrt{m}N \leqslant \sqrt{m}2^{1-m} M^{1/m}$.
\end{proof}

Another fundamental but more technical result is the Khintchine-Groshev Theorem. We do not include the proof here, see the latest improvement by Hussain and Yusupova \cite{HusYus}.

\begin{theorem}[Khintchine-Groshev]
\label{thm:khintchine}Let $\psi :\mathbb  N\to \mathbb{R}_{+}$ tending to $0$ such that 
\begin{align*}
\sum_{q\in \mathbb{Z} ^{m} }\psi ( | q | )^{d}<\infty .
\end{align*}
Then the set of $\bo\in (\mathbb{R}^{m})^{d}$ that are $\psi $-SWA* is $(\Leb[m])^{d}-$negligible.  If on the other hand the sum diverges then $(\Leb[m
])^{d}-$a.a. $\bo \in (\mathbb{R}^{m})^{d}$ is $\psi $-SWA*, in the case $m=d=1$ $\psi $ needs furthermore to be monotonic.
\end{theorem}

The theorem yields that  $(\Leb[m
])^{d}-$a.a. $\bo$ has irrationality index $\tau (\bo)=m/d$, where the irrationality index of some $\bo \in (\mathbb{R}^{m})^{d}$ is defined by
\begin{align*}
\tau  (\bo ):=\inf\{\tau  :\bo \text{\rm{ is not $\tau   $-SWA*}}\}=\sup\{\tau  :\bo \text{\rm{ is $\tau $-SWA*}}\}.
\end{align*}

%
%
%

In particular, for $\Leb[m]$-a.a. $\omega \in \mathbb{R}^{m}$, $\omega $ is $(m+\eta )-$BA for each $\eta >0.$
The following result yields that the situation is  the same if SWA* is replaced by SWA. Actually, for most statements about the quantity of existing $\psi $-SWA* arrays, there are about as many $\psi $-SWA arrays. More precisely, we show that for every $\bo$ that is $\psi $-SWA*, there is a $\psi $-SWA array $\bo'$ in the finite neighbourhood
\begin{align*}
\mathcal  N(\bo)=\{\bo'=(\omega' _{[k]})_{k=1}^{d}:\;(\exists k : \omega _{[k]}'=2\omega _{[k]}\text{\rm{  or  }}\exists i: \omega '_{[1],i}=\omega _{[1],i}+1\text{\rm{  or  }}\bo'=\bo)\}.
\end{align*}

\begin{proposition}
\label{prop:psi-SWA*}
Let $\bo \in (\mathbb{R}^{m})^{d}$ that is $\psi $-SWA*  for $\psi :\mathbb  N\to \mathbb{R}_{+}$ non-increasing, then there is $\bo'$ in $\mathcal  N(\bo)$ that is $\psi $-SWA.
\end{proposition}

\begin{proof}

Either $\bo$ is $\psi $-SWA or
there are by definition $c>0$ and  infinitely many distinct $p_{[k]}=p_{[k]}^{j}\in \mathbb{Z}  ,q=q ^{j}\in \mathbb{Z} ^{m},j\geqslant 1,1\leqslant k\leqslant d$ such that $$\sum_{k}(p_{[k]} +\sum_{i}q_{i})=\sum_{k}p_{[k]}+d\sum_{i}q_{i}\equiv 0$$ and $$ | p_{[k]} -\omega _{[k]}\cdot q  | <c\psi ( | q  | ),1\leqslant k\leqslant d.$$  

Let $m_ {j}\in \mathbb{N}$ maximal such that $2^{m_ {j}}$ divides each $p_{[k]},1\leqslant k\leqslant d$ and each $q_{i},1\leqslant i\leqslant m$, and let $\tilde p_{[k]}=2^{-m_ {j}}p_{[k]},\tilde q_{i}=2^{-m_ {j}}q_{i}$.  Since $\psi $ is non-increasing and $ | \tilde q | \leqslant  | q | $, 
\begin{align*}
 | \tilde p_{k}-\omega _{[k]}\cdot \tilde q | =2^{-m_ {j}} |   p_{[k]}-\omega _{[k]}\cdot q |< 2^{-m_ {j}}c\psi ( | q | )\leqslant c\psi ( | \tilde q | ), 1\leqslant k\leqslant d.
\end{align*}
It is important to precise that there are infinitely many pairwise distinct couples $(\tilde p^{j},\tilde q^{j})$ with $\tilde p^{j}=(\tilde p_{[k]}^{j})_{k}$, otherwise there is $j_{0}$ and $m_{j}'\to \infty $ such that for infinitely many $j$, $p^{j}=2^{m_{j}'}\tilde p^{j_{0}},q^{j}=2^{m_{j}'}\tilde q_{j_{0}}$, which contradicts $ | p^{j}_{[k]}-\omega _{[k]}\cdot q^{j} | \to 0.$

If there are infinitely many couples $(\tilde p,\tilde q)\equiv 1$, then $\bo$ is $\psi $-SWA and the proof is complete. Hence let us suppose  in the following that there are infinitely many couples $(\tilde p,\tilde q)\equiv 0$. The maximality of $m_{j}$ and the drawer principle then yield that there is either  $k_{0}$ such that for infinitely many couples $(\tilde p,\tilde q)$, $\tilde p_{[k_{0}]}\equiv 1$, or $i_{0}$ such that for infinitely many couples, $\tilde q_{i_{0}}\equiv 1$.

In the   case where $\tilde p_{[k_{0}]}\equiv 1$, let $$\omega '_{[k],i}=\begin{cases}
2\omega _{[k_{0}],i}$ if $k=k_{0}\\
\omega _{[k],i}$ otherwise$,
\end{cases}p'_{[k_{0}],i}=\begin{cases}
2\tilde p_{[k_{0}],i}$ if $k=k_{0}\\
\tilde p_{[k],i}$ otherwise$,
\end{cases}1\leqslant i\leqslant m.$$
We have $ | p_{[k]}'-\omega '_{[k]}\cdot \tilde q | \leqslant 2\psi ( | \tilde q | ),1\leqslant k\leqslant d$ for infinitely many couples $( p ',\tilde q)$, and $(p',\tilde q)=(\tilde p,\tilde q)+(p_{[k_{0}]i},0)\equiv 1$, hence $\bo':=(\omega _{[k]}')_{1\leqslant k\leqslant d}$ is $\psi $-SWA.

In the case where $\tilde q_{i_{0}}\equiv 1$, let $$\omega '_{[k],i }=\begin{cases}
\omega _{[1],i }+1$ if $k=1,i=i_{0}\\
\omega _{[k],i}$ otherwise$\end{cases}
,p'_{[k]}=\begin{cases}p_{[1]}+q_{i_{0}}$ if $k=1\\
p_{[k]}\text{\rm{  otherwise }}
\end{cases}
.$$ Then 
\begin{align*}
p'_{[k]}-\omega '_{[k]}\cdot \tilde q=\begin{cases}\tilde p_{[1]}+\tilde q_{i_{0}}-\omega _{[1]}\cdot \tilde q-\tilde q_{i_{0}}=\tilde p_{1}-\omega _{1}\cdot \tilde q$ if $k=1\\
\tilde p_{[k]}-\omega _{[k]}\cdot \tilde q$ otherwise$
\end{cases}
\end{align*}
and
\begin{align*}
(p',\tilde q)\equiv (0,\tilde q_{i_{0}}) \equiv 1
\end{align*}
 for infinitely many couples $(p',\tilde q)$, hence $\bo'$ is $\psi $-SWA.
\end{proof}
The next result is useful for tensorizing variance estimates.
 \begin{proposition}
 \label{prop:bold-omega-SWA}
 If $\omega \in \mathbb{R}^{m}$ is $\psi $-WA, $\bo=(\omega ,\dots ,\omega )\in (\mathbb{R}^{m})^{d}$ is $\psi $-SWA.
\end{proposition}

 \begin{proof}
Since $\omega $ is assumed to be $\psi $-WA, there is a sequence $(p^{j},q^{j})_{j}$ such that $ | p^j -\omega \cdot q^j   | <c_{W}\psi ( | q^{j} | )$ and $(p^j ,q^{j} )\equiv 1$. 
Hence with  $q^{j}_{[k]}=q^j ,p^{j}_{[k]}=p^j $, $$ | p^{j}_{[k]}-q^{j}_{[k]}\cdot \omega  | = | p^j -q^j \cdot \omega  |<c_{W}\psi (q^j )$$ but $\sum_{k=1}^{d}(p^{j}_{[k]}+\sum_{i=1}^{m}q^{j}_{[k],i})=d(p^j +\sum_{i}q_{i}^j )$ is odd only if $d$ is odd. If $d$ is even, choose instead $p^{j}_{[1]}=2p^j ,q^{j}_{[1]}=2q^j $, so that $ | p^{j}_{[1]}-q^{j}_{[1]   }\cdot \omega  | <2c_{W}\psi ( | q^j |  )$, and $\sum_{k=1}^{d}(p^{j}_{[k]}+\sum_{}q^{j}_{[k],i})=(2d+1)(p^j +q^j )$ is indeed odd.  This sequence demonstrates that $\bo$ is $\psi $-SWA.

\end{proof}

\subsection{Variance bounds}

Let us start by the proof of Theorem \ref{thm:general-var}, since it will be used for the proof of Theorem \ref{thm:var-cancel} (and Proposition  \ref{prop:structure}).

\subsubsection{Proof of Theorem \ref{thm:general-var}}
\label{sec:prf-main-var}

The starting point is the following lemma, straightforward consequence of  \cite[Lemma 2]{BST}.
\begin{lemma}
\label{lm:cov-ind}
We have for every $u\in \mathbb{R}$ coefficients $\alpha_{n,u}\geqslant 0,n\in \mathbb{N} $ such that for two centred standard Gaussian variables $X,Y$ with correlation $\rho $
\begin{align}
\label{eq:cov-indic}
\Gamma _{u}(\rho ):=\text{\rm{Cov}}( \mathbf{1}_{\{X>u\}},\mathbf{1}_{\{Y>u\}})=\sum_{n=1}^{\infty }\alpha_{n,u}\rho ^{n} =\frac{1}{2\pi }\int_{0}^{\rho }\frac{1}{\sqrt{1-r^{2}}}\exp\left(
-\frac{u^{2}}{1+r}
\right)dr
\end{align}
in particular,  $ \Gamma_{0}(\rho )=\text{\rm{arcsin}}(\rho ) $  with $\alpha_{2n,0}=0$ and $$\alpha _{2n+1}:=\alpha_{2n+1,0}=\frac{{n\choose 2n}}{4^{n}(2n+1)}=\Theta n^{-3/2}.$$
We also have $\alpha _{2,u}\neq 0$ for $u\neq 0$.
\end{lemma}
Let $\U_{n}=\sum_{i=1}^{n}X_{i}$ where the $X_{i}$ are independent and identically distributed with law $\bmu .$ 
Denote by $\gamma ^{\star 2}$ the auto-convolution of $\gamma $ with itself, and by $\bmu ^{n}$ the law of $\U_{n}$.
We have by Lemma \ref{lm:cov-ind}  
\begin{align}
\notag
V_{\bmu }^{\gamma }(T)=&\int_{(\mathbb{R}^{d})^{2}}\Gamma _{u}(C(\t-\s))\gamma  (\t/T)\gamma  (\s/T)d\t d\s \\
\notag=&\int_{(\mathbb{R}^{d})^{2}}\Gamma _{u}(C(\z))\gamma \left(
\frac{\z+{\bf w}}{2T}
\right)\gamma  \left(
\frac{{\bf w}-\z}{2T}
\right)d{\bf w}d\z\\
\notag=&\int_{\mathbb{R}^{d}}\Gamma _{u}(C(\z))\gamma _{2T}^{\star 2}(2\z)d\z\\
\notag=&\sum_{n\in \mathbb{N}}\alpha _{n,u}\int_{}C(\z)^{n}\gamma _{ 2T}^{\star 2}(2\z)d\z\\
\notag=&\sum_{n\in \mathbb{N}}\alpha _{n,u}\int_{} \bmu ^{n}(d\z)\hat\gamma _{ 2T}(2\z)^{2}d\z\text{\rm{ using \eqref
{eq:def-spectral}}}\\
\notag=&\sum_{n\in \mathbb{N}}\alpha _{n,u}\int_{}\bmu ^{n}(d\z)(2T)^{2d}\hat \gamma  (4T\z)^{2}d\z\\
\label{eq:var-comput}
=&(2T)^{2d}\sum_{n\in \mathbb{N}}\alpha _{n,u}\mathbb{E}(\hat\gamma  (4T\U_{n})^{2})\\
\notag=& 2^{2d}(v_{T} ^{(1)}+v_{T} ^{(2)})
\end{align}
where, with $A_{1}=[0,r],A_{2}=]r ,\infty ]$
\begin{align*}
v_{T}^{(i)}=T^{2d}\sum_{n\in \mathbb{N}} \alpha _{n,u}\mathbb{E}(\hat\gamma (4T\U_{n})^{2}\mathbf{1}_{\{\|T\U_{n}\|\in A_{i}\}}).
\end{align*}

  For the case $u\neq 0$, point (iii) simply comes by lower bounding by the term corresponding to $n=2$.

Let us now focus on the case $u=0$. Remark first that since $\bmu$ is $\mathbb{Z} $-free, $\mathbb{P}(\U_{n}=0)=0$ for $n$ odd, hence
\begin{align*}
v_{T}^{(1)}\geqslant  &T^{2d}c_{r}^{-}\K(rT^{-1})
\end{align*}
hence \eqref{eq:var-lower} is proved.
For the second point, the hypothesis on $\gamma $ yields
\begin{align*}
v_{T}^{(2)}\leqslant & T^{2d}\sum_{n\odd}\alpha _{n}\mathbb{E}( c_{1}^{2}\|T\U_{n}\| ^{-d-1}\mathbf{1}_{\{\|T\U_{n}\|>r\}}) \\
=& c_{1}^{2}T^{2d}\sum_{n\odd}\alpha _{n}\int_{0}^{r^{-d-1} }\mathbb{P}((T\|\U_{n}\|)^{-d-1}>y)dy\\
=& c_{1}^{2} T^{2d}T^{-d-1}\sum_{n\odd}\alpha _{n}\int_{0}^{(T/r)^{d+1}} \mathbb{P}(\|\U_{n}\|<y^{-\frac{1}{d+1}}) dy\\
=&c_{1}^{2}  T^{d-1}\int_{0}^{(T/r)^{d+1}}\K (y^{-\frac{1}{d+1}})dy.
\end{align*}
To conclude the proof of (ii), let us assume that $c_{-}\varepsilon ^{\alpha }\leqslant \K (\varepsilon )\leqslant c_{+} \varepsilon ^{\alpha }$ as $\varepsilon \to 0$ for some $0<c_{-}\leqslant c_{+}<\infty $, and let us prove that $\alpha \leqslant d+1$. If $\alpha >d+1$, then the second term on the right hand side of \ref{eq:var-upper} is negligible with respect to the first one (recall that $\K$ is uniformly bounded), we have in particular 
\begin{align*}
c_{r}^{-}c_{-}r^{\alpha } \leqslant T^{\alpha -2d}V_{\bmu}(T) \leqslant c_{r}^{+}c_{+}r^{\alpha }+o_{T\to \infty }(T).
\end{align*} 
Since this is true for all $r$, we have in particular for $0<r_{1}<r_{2}$
\begin{align*}
c_{r_{1}}^{-}c_{-}r_{1}^{\alpha }\leqslant c_{r_{2}}^{+}c_{+}r_{2}^{\alpha }
\end{align*}
which is impossible if we let $r_{2}$ go to $0$ faster than $r_{1}$.

\subsubsection{Proof of Theorem \ref{thm:var-cancel}}
 \label{sec:proof-var-cancel}

We study the field at the original scale $X_{d}$, it is then straightforward to deduce the results for $X_{d,T},T>0.$
We need to estimate $\mathbb{P}(\|\U_{n}\|\leqslant \varepsilon )$ for $n\geqslant 1$, where $ \U_{n}$ is the random walk which increment measure is $ \bmu _{d}$. Equation \eqref{eq:rpw-cov} and the universal bound $|B_{a}(\t)|\leqslant \Theta \|\t\|^{-1/2},\t\in \mathbb{R}^{d}$ yield $ | C_{d}(\t) | \leqslant \Theta (1+ \| \t \| )^{-\frac{ d-1}{2}}$. Then  
\begin{align*}
\mathbb{P}(\|\U_{n}\|\leqslant \varepsilon )=\int_{}\mathbf{1}_{\{B_{d }(0,\varepsilon )\}}(\z)\bmu_{d} ^{n}(d\z)&\leqslant \varepsilon ^{d}\int_{}\mathbf{1}_{\{B_{d }(0,1)\}}(\x)\left|
\int_{\mathbb{R}^{d}}C(\t)^{n}e^{i\varepsilon \x\t}d\t
\right| d\x\\
&\leqslant \Theta \varepsilon ^{d}\int_{\mathbb{R}^{d}}(1+\|\t\|)^{-n\frac{ d-1}{2}}d\t,
\end{align*}
hence for some $c_{5+}<\infty ,$ $\mathbb{P}(\|\U_{n}\|\leqslant \varepsilon )\leqslant c_{5+} \varepsilon ^{d}<\infty $ for $n\geqslant 5$ (and $ n\geqslant 3 $ if $ d\geqslant 4 $).
 We still have to deal with $ 1\leqslant n\leqslant 4 $, and independently with the lower bounds. Let us analyse  the self-convoluted measures $\bmu _{d}^{n},n\geqslant 1$. They are related by the recurrence relation, based on the isotropy of the measures $\bmu^{n} ,n\geqslant 1,$ 
\begin{align}
\notag
\bmu _d^{n+1}(B_{d}(0,r))=&\int_{\bS^{d-1}\times \mathbb{R}^{d}}\mathbf{1}_{\{s+x\in B_{d}(0,r)\}}\bmu _{d}^{n}(ds)\bmu _{d}(dx)\\
\label{eq:mun-recurrence}=&\bmu _{d}(\bS^{d-1})\bmu _d^{n}(\{x:x+e_{1}\in B_{d}(0,r)\})\\
=&\bmu _d^{n}(B_{d}(-e_{1},r))=\bmu _d^{n}(B_{d}(e_{1},r))
\end{align}
where $e_{1}$ is some vector of $\bS^{d-1}$, e.g. $e_{1}=(1,0,\dots ,0).$ 
Let $r>0.$
Hence\begin{align}
\label{eq:mu2}
\bmu_{d} ^{2}(B_{d}(0,r)) =\bmu _{d}( B_{d}(e_{1},r )) 
\end{align}
which is equivalent to $\Leb[d-1](B_{d-1}(0,r))$ as $r \to 0$. Theorem \ref{thm:general-var}-(iii) implies in the case $u\neq 0$  that $$V_{\bmu}^{\gamma,u}(T)\geqslant 2^{2d} c_{r}^{-}\alpha _{2,u} T^{2d}\mu _{d}^{2}(B_{d}(0,rT^{-1}) )\geqslant c_{2,-}T ^{d+1},\varepsilon >0$$
for some $c_{2,-}>0$ (for $r>0$ sufficiently small, see Example \ref{ex:sphere}). 

Since $\alpha _{2,0}=\alpha _{4,0}=0$ (Lemma \ref{lm:cov-ind}),
to treat the case $u=0$ it  remains to study $\bmu _{d}^{3}$ (only for $ d=2$ and $ d=3$). Using \eqref{eq:mun-recurrence}-\eqref{eq:mu2} easily yields constants $c_{3,-},c_{3,+}$ such that $0<c_{3,-}\varepsilon ^{d}\leqslant \bmu _{d}^{3}(B_{d}(0,\varepsilon ))\leqslant c_{3,+}\varepsilon ^{d}<\infty $ as $\varepsilon \to 0$.
Hence
\begin{align*}
\alpha _{3}c_{3,-}\varepsilon ^{d}\leqslant \K (\varepsilon )=\sum_{n\geqslant 3,n\odd} \alpha _{n}\mathbb{P}(\|\U_{n}\|\leqslant \varepsilon )\leqslant \alpha _{3}c_{3,+}\varepsilon ^{d}+\sum_{n=5}^{\infty }\alpha _{n}c_{5+}\varepsilon ^{d}
\end{align*}
gives the desired upper and lower bounds for $u=0$ (using Theorem \ref{thm:general-var}-(i),(ii)).

 \subsubsection{Proof of Proposition \ref{prop:isotropic}}
 \label{sec:isotropic}
 
 The statement in the case $\bmu^{n}(B(0,\varepsilon ))\geqslant c\varepsilon ^{d}, n$ odd, follows immediately from \eqref{eq:var-lower}. If $x\neq 0$, we have $\mu ^{n+m}(B_{d}(0,\varepsilon ))\geqslant c'\varepsilon ^{d}$ for $m> | x | /\varepsilon $ even, for some $c'>0.$
 Let us prove that this is the case if $\bmu$ is isotropic. 
 There is $b>0$ such that $\bmu(A_{b})>0$ where
\begin{align*}
A_{b}=\{\x\in \mathbb{R}^{d}:\|\x\|\in [b,b+1]\}.
\end{align*}
 Up to lower bounding $\bmu$ by $\bmu\mathbf{1}_{A_{d}}$, assume without loss of generality that $\bmu$'s support is contained in $A_{b}$. By isotropy there is a measure $\nu $ on $[b,b+1]$   such that $\bmu $ can be decomposed in $\bmu=\bmu_{d}\times \nu $ in polar coordinates, where $\bmu_{d}$ is the uniform measure on the $d$-dimensional sphere (see Section \ref{sec:proof-var-cancel}). 
 We have
\begin{align*}
C(\t)=&\int_{\mathbb{R}^{d}}\exp(-i\x\cdot \t)\bmu(d\x)\\
=&\int_{b}^{b+1}B_{0}(r\|\t\|)\nu (dr)\\\leqslant& \int_{b}^{b+1}\Theta (1+r\|\t\|)^{-\frac{ 1}{2}}\nu (dr)\\
\leqslant& \Theta (1+b\|\t\|)^{-\frac{ 1}{2}},\t\in \mathbb{R}^{d}.
\end{align*}
 It follows that $C^{2d+1}\in L^{1}(\mathbb{R}^{d})$, hence $\bmu^{2d+1}$ has a bounded continuous density $f$, there is in particular $\x\in \mathbb{R}^{d},r>0,c>0$ such that $f\geqslant c>0$ on $B(\x,r)$. For $m\geqslant 1$, $\mu ^{(2d+1)m}$ hence has a positive density on $B(\x,mr)$, and for $m$ sufficiently large, $\mu ^{(2d+1)m}$ has a positive density on $B(0,1)$. Since $n=(2d+1)m$ is odd for $m$ odd, we indeed have $\mu ^{n}(B_{d}(0,\varepsilon ))\geqslant c'\varepsilon ^{d}$ for some $c'>0.$

 \subsubsection{Proof of Proposition \ref{prop:structure}}
 \label{sec:prf-structure}
 
 Recall from Lemma \ref{lm:cov-ind} that 
\begin{align*}
 \text{\rm{Cov}}(\mathbf{1}_{\{0\in E\}},\mathbf{1}_{\{\t\in E \}})=\sum_{n\odd}\alpha _{n}C(\t)^{n},\t\in \mathbb{R}^{d}
\end{align*}where $ C$ is the reduced covariance function of $ X_{\bo}$. Hence we are looking for $ \mathcal S$ satisfying for $ \varphi $ smooth with compact support
\begin{align*}
\int_{ \mathbb{R}^{d}}\hat \varphi (\x) \mathcal S(d\x)=&\int_{ \mathbb{R}^{d}}\varphi (\t)\sum_{n\odd}\alpha _{n}C (\t)^{n}d\t\\
=& \int_{ \mathbb{R}^{d}}\varphi (\t)\sum_{n\odd}\alpha _{n}\left(
\int_{ \mathbb{R}^{d}}e^{i\t\cdot \x}\bmu(d\t) 
\right) ^{n}d\x
\end{align*} hence 
\begin{align*}
\mathcal S=\sum_{n\odd}\alpha _{n}\bmu ^{n}.
\end{align*}

  Then using  \eqref{eq:Ib0-upper} in the context of Example \ref{ex:standard},
\begin{align*}
\S(B_{d}(0,\varepsilon ))\leqslant &c\sum_{n\odd}\alpha _{n}\bmu ^{n}(B_{d}(0,\varepsilon ))\leqslant c'\varepsilon ^{\frac{ 1+d(m+1)}{m+\eta }}.
\end{align*}
 
\subsection{Proof of Theorem \ref{thm:dioph-rw}}

{\bf Notation.} 
We specify here the notation $A=\Theta B$ to indicate that there are finite  constants $c,c'>0$ depending  on $ m,d,\psi ,\beta   $   and not (further) on $\bo,\varepsilon ,T,n$ such that $A\leqslant cB,B\leqslant c'A$.

Also, for a $d$-tuple of vectors of $\mathbb{R}^{m+1}$, $\bar \x=(\bar x_{[1]},\dots ,\bar x_{[d]})\in (\mathbb{R}^{m+1})^{d} $ with $\bar x_{[k]}=(x_{[k],0},\dots ,x_{[k],m})\in \mathbb{R}^{m+1}$, remove the bar when the $0$-th component is removed from each vector: 
\begin{align*}
x_{[k]}=(x_{[k],1},\dots ,x_{[k],m}),\hspace{1cm}\x=(x_{[1]}\dots ,x_{ [d] }).
\end{align*}
Euclidean norms in $\mathbb{R}^{m}$ are denoted by a single bar and in $(\mathbb{R}^{m})^{d}$ by two bars: 
\begin{align*}
|x_{[k]}|^{2}=\sum_{i=1}^{m}x_{[k],i}^{2},\;\|\x\|^{2}=\sum_{k=1}^{d}|x_{[k]}|^{2}.
\end{align*}
We also define for $q\in \mathbb{Z} ^{m},\omega \in \mathbb{R}^{m}$ 
\begin{align*}
\d_{q}(\omega )=\inf_{p\in \mathbb{Z} } | p-q\cdot \omega  | .
\end{align*}

\begin{lemma}
\label{lm:Ie-dispo} 
Let $\omega\in \mathbb{R}^m  $ that is $\psi $-BA.
For $ 1/2\geqslant \varepsilon >0$, define 
\begin{align*}
I_{\varepsilon }(\omega ) :=\{q\in \mathbb{Z} ^m \setminus  \{0\}:0<\d_{q}(\omega )\leqslant  \varepsilon \}.
\end{align*} 
Let $q^{(N)},N\in \mathbb{N}^{*} $, the elements of $I_{\varepsilon } $ ordered by increasing radius. Then
\begin{align}
\label{eq:qN} | q^{(N)} | \geqslant  \Theta N^{\frac{1}{m}}\psi ^{-1}(\varepsilon ).
\end{align}
 In particular, we prove the following estimate:
\begin{align}
\label{eq:dispo}
\sum_{q\in I_{\varepsilon }}\exp(-\Theta n^{-1}|q|^{2})\leqslant \sum_{N=1}^{\infty }\exp(-\Theta n^{-1}N^{\frac{2}{m}}\psi  ^{-1}(\varepsilon )^{2})\leqslant \Theta n^{\frac{m}{2}}\psi^{-1}(\varepsilon )^{-m}.
\end{align} 
\end{lemma}

\begin{proof} 
The starting point is that for $q\in I_{\varepsilon } $, since $\varepsilon>\varepsilon/2 \geqslant \d_{q}(\omega )/2\geqslant    \psi ( | q | )$, we have $ | q  | \geqslant  \rho :=\psi ^{-1}(\varepsilon ).$
And the triangular inequality yields for $q\neq q '\in I_{\varepsilon } ,$
\begin{align*}
2\psi (q-q')\leqslant \d_{q-q'}(\omega )  \leqslant 2\varepsilon ,
\end{align*}hence  $
 | q-q' | \geqslant  \rho $ as well. 
It follows that all $q\in I_{\varepsilon }$ are pairwise distant by more than $\rho $, and the  balls $B_{m}(q,\rho /2),q\in I_{\varepsilon }$  are disjoint. Hence for $N_{0}\in \mathbb{N}^*,$ the total $\Leb[m]-$ measure occupied by the $B_{m}(q^{(N)},\rho/2 ),N\leqslant N_{0}$ is larger than $\Theta N_{0}\rho ^m $. This volume is necessarily smaller than the volume of the ball with  radius $ | q ^{(N_{0})} | +\rho /2\leqslant 2 | q ^{(N_{0})} | $, hence 
\begin{align*}
\Theta N_{0}\rho ^m    \leqslant \Theta  | q ^{(N_{0})} | ^m 
\end{align*}
which yields  \eqref{eq:qN}. Finally
\eqref{eq:dispo}
follows from
\begin{align*}
\sum_{N =1}^{\infty }\exp(-\Theta n^{-1}(N ^{\frac{1}{m} }\psi  ^{-1}(\varepsilon ))^{2})  &\leqslant 2   \int_{1/2}^{\infty }\exp(-\Theta (n^{-\frac{m}{2}}\psi  ^{-1}(\varepsilon )^{m}y)^{\frac{2}{m} })dy\\
&\leqslant \Theta n^{\frac{m}{2}}\psi  ^{-1}(\varepsilon )^{-m}.\\
\end{align*}
\end{proof}

\begin{proof}[Proof of Theorem \ref{thm:dioph-rw}] Let $M=d(m+1).$
The proof is based on the study of the symmetric random walk $(S_{n})_{n}$ on $\mathbb{Z} ^{M}$ with independent increments defined by $S_{0}=0$ and
\begin{align*}
\mathbb{P}(S_{n+1}=S_{n}\pm \e_{j})=\frac{1}{2M},1\leqslant j\leqslant M,
\end{align*}
where $(\e_{j})_{j}$ is some basis of $\mathbb{R}^{d}$.
 Following the notation introduced above, denote also $ \bar \omega _{[k]}=(1,\omega _{[k]})$ and $ \bar \bo=(\bar \omega _{[k]})_{k}.$

 For $\bar q_{[k]}=(q_{[k],0},q_{[k]})\in \mathbb{Z} ^{m+1},1\leqslant k\leqslant d,\bar \q=(\bar q_{[1]},\dots ,\bar q_{[d]})\in (\mathbb{Z} ^{m+1})^{d}\approx\mathbb{Z} ^{M}$,   denote by  $q_{[k]}=(q_{[k],1},\dots ,q_{[k],m})\in \mathbb{Z} ^{m},\q=(q_{[k]})_{k}\in (\mathbb{Z} ^{m})^{d}$.  
We define
\begin{align*}
 \bar \q\otimes \bar \bo := \q_{0}-(\bar q_{[k]}\cdot   \bar \omega _{[k]})_{k=1,...,d}\end{align*}where $\q_{0}=(q_{[k],0})_{k=1}^{d}$, so that we have the equality in law
$
\U_{n}\equlaw S_{n}\otimes \bar \bo.$

 We use the notation, for $\x=(x_{[1]},\dots ,x_{[d]})\in \mathbb{Z} ^{d},K\subset \llbracket d\rrbracket $,
  $$ \bar \I^{\x,K}_{\varepsilon }(\bo)= \bar \I^{\x,K}_{\varepsilon } =\{\bar \q\in  \mathbb{Z} ^{M} :\bar q_{[k]}=0,k\notin K\text{\rm{ and }}0< |  \bar q_{[k]}\cdot  \bar \omega _{[k]}-x_{[k]}  | \leqslant \varepsilon,k\in K \} .$$

For $\varepsilon <1/2$, an element $\bar \q\in \bar \I^{\x,K}_{\varepsilon }$ satisfies the following for $k\in K$: \begin{align*}
| q_{[k],0}-q_{[k]}\cdot \omega _{[k]}- x_{[k]} | <\varepsilon ,
\end{align*}
hence since $ x_{[k]}\in \mathbb{Z} $ and $q_{[k]}\in I_{\varepsilon }(\omega _{[k]})$, $\omega _{[k]}\cdot q_{[k]} $ is $\varepsilon $-close to $\mathbb{Z} $. It follows that $q_{[k],0}$ depends explicitly on other coordinates
\begin{align}
\label{eq:jk0}
q_{[k],0}=&q_{[k],0}( x_{[k]},q_{[k]}):=\argmin_{p\in \mathbb{Z} } | p-q_{[k]}\cdot \omega _{[k]}-x_{[k]} |\\
\notag\q_{0}=&\q_{0}(\x,\q):=(q_{[k],0})_{k}.
\end{align}
In particular,  $ | x_{[k]} | \leqslant  | q_{[k],0} | + |q_{[k]}\cdot \omega _{[k]} | +1 $, and
\begin{align}
\label{eq:j2}
 \| \bar \q \| ^{2}=&\|\q_{0}\|^{2}+\|\q\|^{2}\geqslant \max(\|\q\|^{2},\|\q\|^{2}+\Theta( \|\x\|^{2}-1))\geqslant \Theta (\|\q\|^2+\|\x\|^{2}).
\end{align}
We also have the one-to-one correspondance
\begin{align}
\label{eq:1to1}
\I^{\x,K}_{\varepsilon }(\bo):&=\{\q\in \prod_{k\in K}\mathbb{Z} ^{m }: (\q_{0}(\x;\q),\q)\in \bar \I_{\varepsilon }^{\x,K}(\bo) ,k\in K\}\\
&\equiv \{\o\}^{d-|K|}\times \prod_{k\in K}I_{\varepsilon }(\omega _{[k]}).
\end{align}
 
{\bf Proof of (i):}
By the Gaussian approximation   Lemma \ref{lm:Gauss-approx} (below), and \eqref{eq:j2}, \begin{align}
\notag p_{n}^{\x,K}(\varepsilon )&=\sum_{\bar \q\in \bar \I^{\x,K}_{\varepsilon }}\mathbb{P}( S_{n} =\bar \q )=\sum_{  \q\in  \I^{\x,K} _{\varepsilon }}\mathbb{P}( S_{n} = ( \q_{0}(\x,\q),\q ))\\
\notag&\leqslant \Theta \sum_{\q\in \I^{K,\x}_{\varepsilon }} n^{-\frac{M}{2}} {\exp(-\Theta n^{-1}( \| \q \| ^{2}+\|\q_{0}\|^{2}))}\\
\notag &\leqslant \Theta n^{-\frac{M}{2}}  \sum_{\q\in \I^{K,\x}_{\varepsilon }} {\exp(-\Theta n^{-1} \| \q \| ^{2})\exp(-\Theta n^{-1} \|\x\|^2 ))} \\
\label{eq:pnx}&\leqslant \Theta n^{-\frac{M}{2}} \exp(-\Theta n^{-1}\|\x\|^{2})   \prod_{k\in K}\sum_{q_{[k]}\in I_{\varepsilon }(\omega _{[k]})}\exp(-\Theta n^{-1} | q_{[k]} | ^{2} ))\text{\rm{ by \eqref{eq:1to1}}}\\
\notag&\leqslant \Theta n^{-\frac{d(m+1)}{2}}\exp(-\Theta n^{-1}\|\x\|^{2}) (n^{\frac{m  }{2}}\psi^{-1}(\varepsilon )^{-m})^{|K|} \;\;\text{\rm{ with   \eqref{eq:dispo},}}\\
\notag&\leqslant \Theta n^{-d/2}n^{-\frac{(d-|K|) m}{2}}\psi^{-1}(\varepsilon )^{-m|K|}\exp(-\Theta n^{-1}\|\x\|^{2})  
\end{align}
and \eqref{eq:pn0-upper} is proved.

The bound  \eqref{eq:pn-upper}
immediately stems from $\bar p^{K}_{n}=\sum_{\x\in \mathbb{Z} ^{d}} p_{n}^{\x,K}$ and Lemma \ref{lm:Gauss-sum} (after summing over $i\in \{0,1\}$). 
Hence  using \eqref{eq:dispo}, and \eqref{eq:pnx} with $\x=0$
\begin{align*}
\I_{\beta }^{\o}(\varepsilon ):=&\sum_{n\geqslant n_{\varepsilon } }n^{-\beta/2 }p_{n}^{\o}(\varepsilon )\\
=&\sum_{n\geqslant n_{\varepsilon }}n^{-\beta /2}\sum_{K\neq \emptyset }p_{n}^{\o,K}(\varepsilon )\\
\leqslant&   \Theta {\bf K}_{\beta}(\varepsilon )\\
\text{\rm{ with }}{\bf K}_{\beta }(\varepsilon ):=&    \sum_{n\geqslant n_{\varepsilon } }n^{-\beta/2 -\frac{M}{2}}\sum_{K\subset \llbracket d\rrbracket ,K\neq \emptyset } \prod_{k\in K}  \sum_{N_{k}=1}^{\infty }\exp(-\Theta n^{-1}(N_{k}^{\frac{1}{m }}\psi^{-1}(\varepsilon ))^{2})\\
\leqslant &    \sum_{n\geqslant n_{\varepsilon } }n^{-\beta/2 -\frac{M}{2}}\sum_{K\subset \llbracket d\rrbracket ,K\neq \emptyset } \sum_{N_{k}\geqslant 1,k\in K}^{\infty }\exp(-\Theta n^{-1}\sum_{k\in K }(N_{k}^{\frac{1}{m }}\psi^{-1}(\varepsilon ))^{2})\\
\leqslant  &\Theta \sum_{K\subset \llbracket d\rrbracket ,K\neq \emptyset} \sum_{N_{k}\geqslant 1,k\in K}\sum_{n\geqslant n_{\varepsilon } }\int_{n}^{n+1/2}(z-1/2)^{-\beta/2 -\frac{M}{2}}\\
&\hspace{5cm}\times    \exp(-\Theta z^{-1}\sum_{k\in K} \psi^{-1}(\varepsilon )^{2}N_{k}^{2/m } )dz\\
\leqslant & \Theta  \sum_{K\subset \llbracket d\rrbracket ,K\neq \emptyset }\sum_{N_{k}\geqslant 1,k\in K} (\sum_{k\in K}
\psi^{-1}(\varepsilon )^{2}
 N_{k}^{2/m })^{1-\beta/2 -\frac{M}{2}}\\
 &\hspace{5cm}\times \int_{0}^{\infty }y^{\beta/2 +\frac{M}{2}-2}\exp(-\Theta y )dy\\
 \leqslant &\Theta (\psi^{-1}(\varepsilon )^{2})^{1-\beta/2 -\frac{M}{2}}\max_{K\subset \llbracket d\rrbracket ,K\neq \emptyset }
 \int_{[1,\infty  ]^{  | K | }}(\sum_{k\in K}x_{k}^{\frac{2}{m }} )^{1-\beta/2 -\frac{M}{2}}\prod_{k\in K}dx_{k}\\
 \leqslant &\Theta \psi^{-1}(\varepsilon )^{2-M-\beta }\max_{1\leqslant p\leqslant d}\int_{ [1,\infty  ]^{p}}(\sum_{k=1}^{p}y_{k})^{1-\beta/2 -\frac{M}{2}}\prod_{k=1}^{p}y_{k}^{\frac{m }{2}-1}dy_{k}\\
 \leqslant &\Theta \psi^{-1}(\varepsilon )^{2-d(m+1)-\beta }\max_{1\leqslant p\leqslant d}\int_{1}^{\infty }(\Theta r)^{1-\beta/2 -\frac{(m+1)d}{2}}r^{mp/2-p}r^{p-1}dr 
 \end{align*}
and the integral converges if  $\beta/2 >1-d/2$.
Since there are less terms in $\J_{\beta }(\varepsilon )$ than in $\I_{\beta }^{\o}(\varepsilon ),$ the upper bound holds and \eqref{eq:Ib0-upper}  is proved.

With the  same computations, using first \eqref{eq:pnx},  and then\eqref{eq:dispo}, and Lemma \ref{lm:Gauss-sum},
\begin{align*}
\I_{\beta } (\varepsilon )
=&\sum_{n\geqslant n_{\varepsilon }}n^{-\beta /2}\sum_{\x\in \mathbb{Z} ^{d}}\sum_{K\neq \emptyset }p_{n}^{\x,K}(\varepsilon )\\
\leqslant & \sum_{n\geqslant n_{\varepsilon } }n^{-\beta/2 -\frac{M}{2}}\sum_{\x\in \mathbb{Z} ^{d}}\exp(-\Theta n^{-1}\x^{2})\\
&\hspace{3cm}\times \sum_{K\subset \llbracket d\rrbracket ,K\neq \emptyset } \prod_{k\in K}  \sum_{N_{k}=1}^{\infty }\exp(-\Theta n^{-1}(N_{k}^{\frac{1}{m }}\psi^{-1}(\varepsilon ))^{2})\\
\leqslant &  \Theta {\bf K}_{\beta -d}(\varepsilon )
\end{align*}
provided $\beta /2>1$ , which proves \eqref{eq:Ib-upper}.

  Let us conclude with the proof of {\bf (iii)}, the proof of {\bf (ii)} is similar and easier.  There are by hypothesis   infinitely many $q^j\in \mathbb{Z} ^{m},j\geqslant 1$ and $p^j_{[k]}\in \mathbb{Z},1\leqslant k\leqslant d, $ such that $\bar \q^j:=((p^j_{[k]},q^j))_{k}\equiv 1$ and  $$ | p^j_{[k]}-\omega _{[k]}\cdot q^j   |  \leqslant  c_{W}\psi  (|q^j|)=:c_{W}\varepsilon _j$$
(we have $\varepsilon _j\to 0$ because $\psi $ converges to $0$ by hypothesis).
We have in particular  with Cauchy-Schwarz inequality 
$$
\|\bar \q^j\|\leqslant \sum_{k=1}^{d} (| p^j_{[k]} | + | q^j | )\leqslant \sum_{k=1}^{d}( | \omega _{k} |  | q^j | +1+ | q^j | )\leqslant \Theta  | q^j | $$
and clearly the other inequality as well $ | q^j | \leqslant \|\bar \q^j\|$.

Then, by Lemma \ref{lm:Gauss-approx},  with $\tilde n_{j}:= {\cG^{-1} |\bar \q^j |\vee n_{\varepsilon _{j}}}$
\begin{align*}
\J_{\beta}(\varepsilon _j)&=\sum_{n\geqslant  n_{\varepsilon _{j}},n\text{\rm{ odd}}}n^{-\beta/2 }p_{n}^{\o}(\varepsilon _j)\\\geqslant &\sum_{n\geqslant n_{\varepsilon_{j} },n\text{\rm{ odd}}}n^{-\beta/2 }\mathbb{P}(S_{n}=\bar \q^j)\\ \geqslant &\Theta \sum_{n\geqslant \tilde n_{j},n\equiv \bar \q^j  \equiv 1}n^{-\beta/2 }n^{-\frac{d(m+1)}{2}}\exp(-\Theta n^{-1}  \| \bar  \q^j   \| ^{2})\\ 
\geqslant &\Theta \int_{[\tilde n_{j}/2] }^{\infty }y^{-\beta/2 -d\frac{m+1}{2}}\exp(-\Theta y^{-1} |   q^j  | ^{2})dy\\
\geqslant &\Theta| q^j  |^{2-\beta  -d(m+1) }\int_{0}^{\Theta    |  q^j  | ^{2}\tilde n_{j}^{-1}}z^{\beta/2 +d\frac{m+1}{2}-2}\exp(-\Theta z)dz\\
\geqslant & \Theta \psi^{-1}(\varepsilon _j)^{2-\beta  -d(m+1) }
\end{align*}
provided $\beta >0$, because $ |  q^j  | ^{2} \|  \bar\q^j  \| ^{-1}\geqslant \Theta >0$ and $n_{\varepsilon _{j}}\leqslant \psi ^{-1}(\varepsilon _{j})^{2}$ yields  (recalling $\psi (|q^{j}|)=\varepsilon _{j}$)
\begin{align*}
 | q^j  | ^{2}n_{\varepsilon_{j} }^{-1}\geqslant \psi ^{-1}(\varepsilon _{j})^{2}\psi ^{-1}(\varepsilon _{j})^{-2}=1,
\end{align*}
hence \eqref{eq:Jb0-lower} is proved. The proof of \eqref{eq:Ib-lower} is similar without the requirement that $\bar \q^j\equiv 1$, hence the sum is over all $n\geqslant n_{\varepsilon _{j}}$ (even and odd).

\end{proof}

\subsubsection{Gaussian approximation}

The following lemma quantifies how much  $S_{n}$ is close to a Gaussian distribution.

\begin{lemma}
\label{lm:Gauss-approx}
Let $\theta _{0}\in (0,\frac{ 1}{2}), M\in \mathbb{N}$ and $S_{n}$ be the symmetric random walk on $\mathbb{Z} ^{M}$ with weights $\theta _{i}\in (\theta _{0},1-\theta _{0}),1\leqslant i\leqslant M$, summing to $1$, i.e. 
\begin{align*}
\mathbb{P}(S_{n+1}=S_{n}\pm \e_{i})=\frac{\theta _{i}}{2},1\leqslant i\leqslant M,n\in \mathbb{N}.
\end{align*} 
For $\q=(\q_{i})\in \mathbb{Z} ^{M},n\in \mathbb{N}$, write $\q\equiv n$ if $\sum_{i=1}^{M}\q_{i}$ and $n$ have the same parity, and remark that $\mathbb{P}(S_{n}=\q)=0$ if $\q\not\equiv n.$ There is a constant $\cG>0$  such that
for $\q\in \mathbb{Z} ^M
$
\begin{align}
\label{eq:up-bnd-binom}
\mathbb{P}(S_{n}=\q)\leqslant& \Theta n^{-\frac{p}{2}}\exp(-\Theta n^{-1} \| \q \| ^{2})\\
\notag\mathbf{1}_{\{\|\q\|\leqslant \cG n\}}\mathbb{P}(S_{n}=\q)\geqslant& \Theta n^{-\frac{p}{2}}\exp(-\Theta n^{-1} \| \q \| ^{2})\mathbf{1}_{\{|\q|\leqslant \cG \,n\}}\text{\rm{ for }}\q\equiv n.
\end{align}\end{lemma}

\begin{remark}The constants involved in this result depend also on  $ \theta _{0}$.
\end{remark}

\begin{proof} 
\renewcommand{\N}{{\bf N}}
Let $N_{i}$ be the number of times direction $i$ has been chosen in the random walk, and  let $B_{i}\leqslant N_{i}$ be the number of $+\e_{i}$ increments, hence $N_{i}-B_{i}$ is the number of $-\e_{i}$ increments.
The $i$-th component of $S_{n}$ is therefore $S_{n,i}:=2B_{i}-N_{i}$. We  have $N_{i}\sim \mathcal{B}(n,\theta _{i}),B_{i}\sim \mathcal{B}(N_{i},1/2)$, and the $B_{i}$ are independent conditionally on $\N:=(N _{i})_{i}$. Hence for $ | \varepsilon |  \leqslant  \cB$, from Lemma \ref{lm:binomial-approx} 
\begin{align*}
\mathbb{P}(B_{i}=[N_{i}(1/2+\varepsilon)] \;|\;\N)=\Theta \exp(-\Theta N_{i}\varepsilon ^{2})N_{i}^{-1/2}.
\end{align*}
Let $\q=(\q_{i})\in \mathbb{Z} ^M$ such that for $1\leqslant i\leqslant p, | \q_{i} | \leqslant  \cB N_{i}$, let $\varepsilon_{i} =N_{i}^{-1}\q_{i}$,
\begin{align*}
\mathbb{P}(S_{n,i}=\q_{i}\;|\;\N )&=\mathbb{P}(B_{i}=N_{i}/2+\q_{i}/2\;|\;\N )\\
&=\begin{cases}
0$ if $N_{i}\not\equiv \q_{i}\\
  \Theta N_{i}^{-1/2}\exp(-\Theta N_{i}\varepsilon_{i} ^{2})=\Theta N_{i}^{-1/2}\exp(-\Theta N_{i}^{-1}\q_{i}^{2})$ otherwise.$
\end{cases}
\end{align*} 
Let
\begin{align*}
\cG:= {\cB}{(\min_{i}\theta _{i}}-\cB)>0.
\end{align*}
If for all $i$, $N_{i}>(\theta _{i}-\cB)n$ and $ | \q_{i} | <\cG n$, then $ | \q_{i} | <\cB N_{i}$ (and $N_{i}=\Theta n$) and we have the lower bound
\begin{align*}
\mathbb{P}(S_{n}=\q)=&\mathbb{E}(\mathbb{P}(S_{n}=\q \;|\;\N ))\\
=&\mathbb{E}(\mathbf{1}_{\{\q_{i}\equiv N_{i} ,\forall i\}}  \mathbb{P}(S_{n}=\q\;|\;\N ) )\\
\geqslant & \mathbb{E}(\mathbf{1}_{\{\q_{i}\equiv N_{i},N_{i}>(\theta _{i}-\cB)n,\forall i\}}\mathbb{P}(S_{n}=\q\;|\;\N ))\\
\geqslant & \mathbb{E}(\mathbf{1}_{\{\q_{i}\equiv N_{i},N_{i}>(\theta _{i}-\cB)n,\forall i\}}\Theta \prod_{i}N_{i}^{-\frac{1}{2}}\exp(-\Theta N_{i}^{-1}\q_{i}^{2}))\\
\geqslant & \mathbb{E}(\mathbf{1}_{\{\q_{i}\equiv N_{i},N_{i}>(\theta _{i}-\cB)n,\forall i\}}\Theta n^{-\frac{p}{2}}\exp(-\Theta n^{-1}\q^{2}))\\
\geqslant &\Theta n^{-\frac{p}{2}}\exp(-\Theta n^{-1}\q^{2})\mathbb{P}(\q_{i}\equiv N_{i},N_{i}>(\theta _{i}-\cB)n,\forall i).
\end{align*}
Since $\sum_{i}N_{i}\equiv n$,
if we do not have $\sum_{i}\q_{i}\equiv n$, we cannot have $N_{i}\equiv \q_{i}, \forall i.$ Otherwise, asymptotically a fraction $2^{-p}$ of admissible tuples $\N \in \llbracket n\rrbracket ^M$ are such that $N_{i}\equiv \q_{i}, \forall i$, hence $\mathbb{P}( \q_{i}\equiv N_{i},N_{i}>(\theta _{i}-\cB)n,\forall i)=\Theta\mathbf{1}_{\{\q\equiv n\}}\mathbb{P}(N_{i}>(\theta _{i}-\cB)n,\forall i) $ and the latter probability converges to $1$ thanks to Lemma \ref{lm:binomial-approx}
, hence the lower bound is proved.

The upper bound is a bit delicate.  Let us start by the trivial bound, if  $ | \q_{i} | >n$ for some $i$,$$\mathbb{P}(S_{n}=\q)=0\leqslant \Theta n^{-1/2}\exp(-\Theta n^{-1}\q_{i}^{2}).$$

Assume henceforth  that $ | \q_{i} | \leqslant n$ for all $i$. Let $\Omega $ be the event such that for some $i$, $N_{i}<\theta _{i}(1-\cB)n$.
 On $\Omega ^{c}$, $N_{i}=\Theta n$ for all $i$, hence by Lemma \ref{lm:Gauss-approx}
\begin{align*}
\mathbb{P}(S_{n,i}=\q_{i}\;|\;\Omega ^{c} )\leqslant  \Theta n^{-1/2}\exp(-\Theta n^{-1}\q_{i}^{2}).
\end{align*}
Finally, in all cases,
\begin{align*}
\mathbb{P}(S_{n}=\q)\leqslant & \mathbb{E}(\mathbf{1}_{\{\Omega^{c}\}} \prod_{i}\mathbb{P}(S_{n,i}=\q_{i}\;|\;\N ))+\mathbb{P}(\Omega )\\
\leqslant &\mathbb{E}(\mathbf{1}_{\{\Omega^{c}\}} \Theta \prod_{i} n^{-\frac{ 1 }{2}}\exp(-\Theta n^{-1}\q_{i}^{2} ))+\mathbb{P}(\Omega  )\\
\leqslant &n^{-p/2}\exp(-\Theta n^{-1}\q^{2})+\mathbb{P}(\Omega  ).
\end{align*}
 Then   Lemma \ref{lm:binomial-approx} with $\varepsilon =-\cB $ yields, using the decreasing of binomial probabilities around the mean,
\begin{align*}
\mathbb{P}(\Omega )\leqslant& \sum_{i}\sum_{k<[n(\theta _{i}-\cB )]}\mathbb{P}(N_{i}=k)\\
\leqslant &\sum_{i}  n\mathbb{P}(N_{i}=[n(\theta _{i}+\varepsilon )])\\
\leqslant &\Theta n^{1/2}\exp(-\Theta n)\\
\leqslant &\Theta n^{-\frac{p}{2}}\exp(-\Theta n/2)\\
\leqslant &\Theta n^{-\frac{p}{2}}\exp(-\Theta n^{-1}\q^{2}),
\end{align*}
using $ \| \q \| \leqslant n,$ which concludes the proof of \eqref{eq:up-bnd-binom}.
\end{proof}  

\begin{lemma}
\label{lm:Gauss-sum} 
For $i\in \{0,1\}$
\begin{align*}
 \sum_{\x\in \mathbb{Z} ^{d},\x\equiv i}\exp(-\Theta n^{-1}\x^{2})=\Theta n^{d/2}.
\end{align*}
where $\x\equiv i$ means that $\sum_{k=1}^{d}\x_{[k]}$ has the same parity as $i$.
\end{lemma}

\begin{proof}

The lower bound stems from $y^{2}\geqslant \min\limits_{\x\in \mathbb{Z} ^{d}\cap B(y,2),\x\equiv i\text{\rm{ or }}\x=0}\x^{2},y\in \mathbb{R}^{d},$ and
\begin{align*}
\Theta n^{d/2}\leqslant \int_{\mathbb{R}^{d}}\exp(-\Theta n^{-1}y^{2})dy\leqslant& \int_{}\max_{\x\in \mathbb{Z} ^{d}\cap B(y,2),\x\equiv i\text{\rm{ or }}\x=0}\exp(-\Theta n^{-1}\x^{2})dy\\
\leqslant & 4^{d}\sum_{\x\in \mathbb{Z} ^{d},\x\equiv i\text{\rm{ or }}\x=0}\exp(-\Theta n^{-1}\x^{2})\\
\leqslant &4^{d}(\sum_{\x\in \mathbb{Z} ^{d},\x\equiv i}\exp(-\Theta n^{-1}\x^{2})+1)
\end{align*}
because at most a mass $4^{d}$ of $y's$ are within distance $2$ from some $\x\in \mathbb{Z} ^{d}$. For the upper bound,  for $\x\in \mathbb{Z} ^{d}\setminus \{0\},$ there is at least one unit cube $C_{\x}$ with integer coordinates within the $2^{d}$ cubes that touch $\x$
 such that for all $y\in C_{\x}$, $y^{2}\leqslant \x^{2}$. Hence
\begin{align*}
\sum_{\x\equiv i}\exp(-\Theta n^{-1}\x^{2})\leqslant &\sum_{\x\equiv i,\x\neq 0}\exp(-\Theta n^{-1}\x^{2})+1 \\
\leqslant&  \sum_{\x\in \mathbb{Z} ^{d}\setminus\{0\} }\int_{C_{\x}}\exp(-\Theta n^{-1}y^{2})dy+1\\
\leqslant&   2^{d}\int_{\mathbb{R}^{d}}\exp(-\Theta n^{-1}y^{2})dy+1\leqslant \Theta n^{d/2}.
\end{align*}
\end{proof}

\subsubsection{Binomial estimates}

\begin{lemma}
\label{lm:binomial-approx}Let $\theta _{0}<\frac{ 1}{2}$.
There is a constant $\cB\in (0,1)$ depending on $\theta _{0}$ such that
for $\theta \in (\theta _{0},1-\theta _{0}), B\sim \mathcal{B}(m,\theta )$, for $-\cB  \leqslant \varepsilon _{m}=\varepsilon \leqslant \cB $ 
\begin{align*}
\mathbb{P}(B= [m(\theta +\varepsilon )])=&\Theta m^{-1/2} \exp\left(
-\Theta {m\varepsilon ^{2}} \right)
\end{align*}
where the constants involved in $\Theta $ depend on $\theta_{0} $, and not on $\theta ,m,\varepsilon $.
\end{lemma}

\begin{proof}Let   $c_{0}=\min(\frac{ 1}{2},\frac{ 1-\theta _{0}}{2})$, and $\varepsilon \in (-c_{0}\theta ,c_{0}\theta )$. Let then $k=[ m( \theta +\varepsilon )]$.
By Stirling's formula, 
\begin{align*}
\mathbb{P}(B =k)=&\Theta \frac{\sqrt{m}}{\sqrt{k}\sqrt{m-k}}\theta ^{k}(1-\theta) ^{m-k}\frac{m^{m}}{k^{k}(m-k)^{m-k}} \\
=&\Theta m^{-1/2}\frac{\theta ^{k}(1-\theta )^{m-k}}{\sqrt{(\theta+\varepsilon ) (1-\theta-\varepsilon  )}}\\
&\hspace{3cm}\times \frac{m^{m}}{(\theta m)^{k}(m(1-\theta ))^{m-k}\left(
\frac{k}{\theta m}
\right)^{k}\left(
\frac{m-k}{m(1-\theta )}
\right)^{m-k}}\\
=&\Theta m^{-1/2}\frac{1}{\sqrt{\theta  }} \left(
 {1+\frac{\varepsilon }{\theta }}
\right)^{-k}\left(
 {1-\frac{\varepsilon }{1-\theta }}
\right)^{k-m}\\
=&\Theta m^{-1/2}\theta _{0}^{-1/2}\exp(\gamma _{\varepsilon ,\theta })\end{align*}
where  
\begin{align*}
\gamma _{\varepsilon ,\theta }=&-m(\theta +\varepsilon )( \frac{\varepsilon }{\theta }-\frac{\varepsilon ^{2}}{2\theta ^{2}}+O(\varepsilon ^{3} ))\\
&\hspace{3cm}-m((1-\theta )-\varepsilon )( -\frac{\varepsilon }{1-\theta }+\frac{\varepsilon ^{2}}{2(1-\theta )^{2}}+O(\varepsilon ^{3}))\\
=&m\frac{\varepsilon ^{2}}{2\theta  }-\frac{m\varepsilon ^{2}}{\theta }+O(m\varepsilon ^{3}+m\varepsilon ^{4})-\frac{m\varepsilon ^{2}}{2(1-\theta )}-\frac{m\varepsilon ^{2}}{1-\theta }+O(m\varepsilon ^{3}+m\varepsilon ^{4})\\
=&\frac{-m\varepsilon ^{2}}{2\theta }-\frac{3m\varepsilon ^{2}}{2(1-\theta )}+O(m\varepsilon ^{3})\\
=&-\Theta m\varepsilon ^{2}
\end{align*}for $ | \varepsilon  | $ sufficiently small.\end{proof}

\subsection{Proof of Theorem \ref{thm:main-intro}}
\label{sec:proof-main-psi}
To prove the upper bound we first need the following computation.

\begin{proposition}
\label{prop:psi-regular}
Let $\psi(q)=q^{-\tau }L(q) $ ARV (Definition \ref{def:regular}) and assume $\bmu$ is of the form \eqref{mu:general} with each $\omega _{[k]}$ that is $\psi $-BA.  Then as $T\to \infty ,$
\begin{align*}
\max\Large\{T^{2d}\J_{3}(T^{-1}) , 
T^{d-1}&\int_{0}^{T^{d+1} }\J_{3}(y^{-\frac{1}{d+1}})dy\Large\} &\\
&\leqslant \begin{cases}
\Theta T^{d-1}&$ if $\psi (q)\geqslant q^{-\tau ^{*}}\ln(q)^{1/d}\\
\Theta T^{d-1}\ln(T)&$ if $\psi (q)=q^{-\tau ^{*}}\\
 \Theta T^{2d}\psi ^{-1}(T^{-1})^{-1-d(m+1)}&$ if $\tau >\tau ^{*}
\end{cases}
\end{align*}

\end{proposition}

\begin{proof}According to \eqref{eq:Ib0-upper} in Theorem \ref{thm:dioph-rw}-(i), 
\begin{align*}
\J_{3}(\varepsilon )\leqslant \Theta \psi ^{-1}(\varepsilon )^{-1-d(m+1)},\varepsilon >0,
\end{align*}
which yields that $T^{2d}\J_{3}(T^{-1})$ admits an upper bound consistent with the claim.

To deal with the other term, assume without loss of generality that $\psi $ is extended to a  smooth strictly non-increasing function $z^{-\tau }L(z):[a,\infty )\to (0,1]$ for some $a\geqslant 1$, such that $L'(z)=o(z^{-1}L(z))$ (the contribution of the integral on $(0,a)$ is uniformly bounded). Make the change of variables $z=\psi ^{-1}(y^{-\frac{1}{d+1}})$, i.e. $\psi (z)^{-d-1}=y$, let $Z=\psi ^{-1}(T^{-1}).$
\begin{align*}
\int_{a}^{T^{d+1} }\J_{3}(y^{-\frac{1}{d+1}})dy\leqslant &\Theta \int_{a}^{T^{d+1}}\psi ^{-1}(y^{-\frac{1}{d+1}})^{-1-d(m+1)}dy\\
=&\Theta \int_{\Theta }^{Z}z^{-1-d(m+1)}(\psi (z)^{-d-1})'dz.
\end{align*}
The hypotheses on $\psi $ yield
\begin{align}
\notag
(\psi (z)^{-d-1})'=&(d+1)(\tau z^{-\tau -1} L(z)-z^{-\tau }L'(z))\psi (z)^{-d-2}\\
=&(d+1)(\tau z^{-1}\psi (z)-z^{-\tau }o(z^{-1}L(z)))\psi (z)^{-d-2}\\
\label{eq:psi}\sim_{z\to \infty }& (d+1)\tau z^{-1}\psi (z)^{-d-1}.
\end{align}
In the case $\tau \leqslant \tau ^{*}$, the previous two displays yield
\begin{align*}
\int_{a}^{T^{d+1} }\J_{3}(y^{-\frac{1}{d+1}})dy\leqslant \Theta \int_{\Theta }^Zz^{-2-d(m+1)}\psi (z)^{-(d+1)}dz
\end{align*}
and the integral converges if $\psi (q)>q^{-\tau ^{*}}\ln(q)^{1/d}$, and if $\psi (q)=q^{-\tau ^{*}}$ it behaves in $\ln(Z)=\Theta \ln(T)$.

Let us turn to the case $\tau >\tau ^{*}$.  Let $\tau '\in (\tau ^{*},\tau )$, we have by \eqref{eq:psi} as $z\to \infty $
\begin{align*}
(z^{-1-d(m+1)}\psi (z)^{-d-1})'=&z^{-1-d(m+1)}(\psi (z)^{-d-1})'\\
&\hspace{2cm}-(1+d(m+1))z^{-2-d(m+1)}\psi (z)^{-d-1}\\
\geqslant & z^{-1-d(m+1)}(\psi (z)^{-d-1})'\\
&\hspace{2cm}-\frac{ 1+d(m+1)}{d+1}z^{-1-d(m+1)}(\psi (z)^{-d-1})'\\
\geqslant &z^{-1-d(m+1)}(\psi (z)^{-d-1})'(1-\frac{\tau *}{\tau ' })
\end{align*}
which results in 
\begin{align*}
\int_{a}^{T^{d+1} }\J_{3}(y^{-\frac{1}{d+1}})dy&\leqslant \frac{\Theta }{1-\tau ^{*}/\tau '}[z^{-1-d(m+1)}\psi (z)^{-d-1}]_{\Theta }^{Z}\\
&=\Theta \psi ^{-1}(T^{-1})^{-1-d(m+1)}T^{d+1}
\end{align*}
which allows to conclude.
\end{proof}

Apply first Theorem  \ref{thm:general-var}-(ii) to the measure $\bmu $ to have bounds on the variance in terms of the function $\J_{3}$, recalling that $\J_{3}=\Theta \K $ (with $\gamma  $ as the unit ball indicator function, see Example \ref{ex:sphere}). Proposition \ref{prop:psi-regular} yields the upper bound. Then  lower bounds for $\J_{3}$ are derived in Theorem \ref{thm:dioph-rw}-(iii), noticing that $\bo$ is   $\psi $-SWA thanks to Proposition \ref{prop:bold-omega-SWA}. In the case $\psi (q)=q^{-\tau ^{*}}\ln(q)^{\frac{ 1}{d}}$, $\psi ^{-1}$ is not explicit, but we use the bound 
\begin{align*}
\psi ^{-1}(\varepsilon )\geqslant c\varepsilon ^{-\frac{ 1}{\tau ^{*}}} | \ln(\varepsilon ) | ^{\frac{ 1}{d\tau ^{*}}},\varepsilon >0
\end{align*}
for some $c>0.$\\

 Regarding the  non-vacuity of Theorem \ref{thm:main-intro}, if $m=1$ and $\tau >1$,  it is a standard fact in diophantine approximation that the set of $\omega $ that are $\psi $-BA and $\psi $-WA is uncountable when $q\psi(q) $ is non-increasing at infinity, see the seminal construction based on continued fractions by Jarn\`ik \cite{Jarnik} and Proposition \ref{prop:psi-SWA*}.

\subsection{Proof of Proposition \ref{prop:rv}}
\label{sec:rv}

Let us assume that \eqref{eq:srv} holds and implies \eqref{eq:sv}: given $\varepsilon >0$, if for $q$ sufficiently large $ | \frac{ L(q+1)}{L(q)}-1 | \leqslant \varepsilon /q$, for $a>1,$ 
\begin{align*}
\left|
\ln\left(
\frac{ L([aq])}{L(q)}
\right)
\right|&=\left|
 \sum_{q\leqslant p\leqslant [aq]}\ln(1+\varepsilon /p)
\right|=\left|
  \sum_{q\leqslant p\leqslant [aq]}\frac{ \varepsilon }{p}\right | +O\left(
\varepsilon ^{2}\sum_{q\leqslant p\leqslant [aq]}\frac{ 1}{p^{2}}
\right)\\
&\leqslant\left|
 \varepsilon\left[
 \ln([aq])+\gamma +o(1)-(\ln(q)+\gamma +o(1))
\right]\right|+o(1),\\
&\leqslant \varepsilon \ln\left(
\frac{ [aq]}{q}
\right)+o(1)
\end{align*}
hence $L([aq])/L(q)$ indeed converges to $1$.

For the converse, assume $L$ is non-increasing, define  $\varepsilon _{p}\geqslant 0$ by induction so that 
\begin{align*}
L(q)=L(1)\prod_{2\leqslant q\leqslant p}\left(
1+\frac{ \varepsilon _{p}}{p}
\right),
\end{align*}
then \eqref{eq:srv} holds means that $\varepsilon _{p}\to 0.$ In this case, \eqref{eq:sv} holds too:
\begin{align*}
\ln\left(
\frac{  L([aq])}{L(q)}
\right)=\sum_{q\leqslant p\leqslant [aq]}\ln(1+p^{-1}\varepsilon _{p})\leqslant ([aq]-q)q^{-1}\sup_{[q,[aq]]}\varepsilon _{p}\xrightarrow[q\to \infty ]{}\;0.
\end{align*}
If $L$ is non-decreasing and \eqref{eq:srv} holds, replace $\varepsilon _{p}$ by $-\varepsilon _{p}.$ Then \eqref{eq:sv} still holds:
\begin{align*}
\ln\left(
\frac{ L([aq])}{L(q)}
\right)=\sum_{q\leqslant p\leqslant [aq]}\ln(1-p^{-1}\varepsilon _{p})\geqslant -\sum_{q\leqslant p\leqslant [aq]}\frac{ \varepsilon _{p}}{p}+O\left(
\sum_{q\leqslant p\leqslant [aq]}\frac{ 1}{p^{2}}
\right)\xrightarrow[q\to \infty ]{}\;0.
\end{align*}

Let us finally study the particular example
\begin{align*}
L(q)=\prod_{1\leqslant p\leqslant q}(1+\frac{ (-1)^{p}}{p}).
\end{align*}
We have 
\begin{align*}
\ln\left(
\frac{ L([aq])}{L(q)}
\right)=\sum_{q\leqslant p\leqslant [aq]}\ln(1+(-1)^{p}/p)\to 0,
\end{align*}
hence \eqref{eq:sv} is satisfied, but \eqref{eq:srv} is not.

\bibliographystyle{plain}

\end{document}